\documentclass[11pt]{article}
\usepackage[utf8]{inputenc}
\usepackage{tikz}
\usetikzlibrary{shapes,arrows,positioning}
\usetikzlibrary{automata,arrows,positioning,calc}
\usepackage{amsmath}
\usepackage{amssymb}
\usepackage{amsthm}
\usepackage{hyperref}
\usepackage{bbm}
\usepackage{blkarray}
\usepackage{algorithmic}
\usepackage{caption}
\usepackage{subcaption}
\usepackage{float}
\usepackage{stmaryrd}
\usepackage{enumitem}

\usepackage{fancyhdr}
\usepackage{amsopn}

\usepackage{bbm}

\usepackage{booktabs}

\hypersetup{colorlinks,
citecolor=orange,
citebordercolor=magenta,
linkcolor=magenta
}

\providecommand{\keywords}[1]
{
  \small	
  \textbf{{Keywords.}} #1
}

\providecommand{\AMS}[1]
{
  \small	
  \textbf{{AMS subject classifications.}} #1
}

\usepackage[linesnumbered,ruled,vlined]{algorithm2e}

\numberwithin{equation}{section}

\usepackage{todonotes}

\newtheorem{theorem}{Theorem}[section]
\newtheorem{remark}[theorem]{Remark}
\newtheorem{definition}[theorem]{Definition}
\newtheorem{lemma}[theorem]{Lemma}
\newtheorem{proposition}[theorem]{Proposition}
\newtheorem{condition}[theorem]{Condition}

\title{Finite State Graphon Games with Applications to Epidemics}

\author{Alexander Aurell\footnote{Department of Operations Research and Financial Engineering,
  Princeton University, 
  Princeton, NJ 08544 
  (\href{mailto:aaurell@princeton.edu}{aaurell@princeton.edu},
  \href{mailtorcarmona@princeton.edu}{rcarmona@princeton.edu},
  \href{mailto:gokced@princeton.edu}{gokced@princeton.edu},
  \href{mailto:lauriere@princeton.edu}{lauriere@princeton.edu}).}
\and Ren\'e Carmona
    \footnotemark[1]
\and G\"ok\c ce Dayan{\i}kl{\i}
    \footnotemark[1]
\and Mathieu Lauri\`ere
    \footnotemark[1]
}
\date{}

\allowdisplaybreaks
\begin{document}

\maketitle
\begin{abstract}
We consider a game for a continuum of non-identical players evolving on a finite state space. Their heterogeneous interactions are represented by a graphon, which can be viewed as the limit of a dense random graph. The player's transition rates between the states depend on their own control and the interaction strengths with the other players. We develop a rigorous mathematical framework for this game and analyze Nash equilibria. We provide a sufficient condition for a Nash equilibrium and prove existence of solutions to
a continuum of fully coupled forward-backward ordinary differential equations characterizing equilibria. Moreover, we propose a numerical approach based on machine learning tools and show experimental results on different applications to compartmental models in epidemiology.
\end{abstract}

\vskip3mm

\vskip3mm
\AMS{  92D30,   
       49N90,   
       91A13,   
       91A15,   
       62M45.  	
}

\vskip 12pt\noindent
\emph{\textbf{Funding.}}
{This work was done with the support of NSF DMS-1716673, ARO W911NF-17-1-0578, and AFOSR \# FA9550-19-1-0291.}

\section{Introduction}

With the recent pandemic of COVID-19, the importance on how to manage large populations in order to control the evolution of the disease has been understood clearly. How a pandemic plays out is a consequence of the interplay of many complex processes, e.g., disease specific spread mechanisms, the network of social interactions, and society-wide efforts to stop or slow the spread.
As individuals, we have all made choices during the ongoing pandemic about the extent to which we minimize our personal risk of being infected. However, there is a real trade-off between being careful and the pursuit of happiness and, as we all have learned by now, our risk is not only determined by our own vigilance but also by others' choices and our environment. 

In the framework of rational agents\footnote{We will interchangeably use the words agent, player, and individual.}, each individual anticipates the action of their neighbours, neighbours' neighbours, etc., and any other outside influence, then selects their action as a best response to others' actions. In other words, they want optimize their private outcome while taking into account their surrounding environment, which includes other agents' actions. This type of strategic interaction is a non-cooperative game. The communication between agents in the game may be restricted by geography, social circles, and other factors. Moreover, people interact with different intensity, depending on their occupation or personal disposition. Hence, the agents in the game, each with their own predisposition for risk, will act in a wide variety of ways and thus naturally form a heterogeneous crowd.

Consider the game discussed above where all agents anticipate the others' action and then selfishly plays a best response to that. Strategy profiles (the collection of all players' actions) consistent with such behavior are Nash equilibria, \textit{i.e.}, profiles such that no player can profit from a unilateral deviation. Computing Nash equilibria in games with large number of players is particularly hard; however, under some specific assumptions, approximate equilibria can be found by using the mean field games approach developed independently by Lasry-Lions \cite{lasry2006jeux,lasry2006jeux2} and Huang-Malham\'e-Caines \cite{huang2006large}. The approach has found many practical applications, examples include the mathematical modeling of price movement in energy markets; pedestrian crowd motion and evacuation; and epidemic disease spread.

One of the fundamental assumptions in mean field game theory is that agents are indistinguishable and interact identically regardless of with whom they interact. However, in some real-world applications such as the modeling of epidemics, the diversity of individuals and the variation of their interactions are important factors to consider. Examples of such include the effects of travel restrictions, multiple age groups with distinct social behavior and risk profiles, and the spectrum of preexisting health conditions. The aspects listed above require the game to have a non-uniform underlying network structure. Games with a large number of non-identical players can be analyzed with so-called graphon games whenever the network specifying the interactions is dense\footnote{Loosely speaking, a dense network or graph is one in which the number of edges is close to the maximal number of edges. For such graphs, there is a functional limit of the adjacency matrix interpreted as a step function. See \cite{lovasz2012large} for an expos\'e of the theory of such limits.}.

Epidemics are driven by the spread of the disease from infected to susceptible agents. The set of susceptible agents is not necessarily the whole non-infected population, depending on the disease immunity can be gained after exposure. The evidence suggests that the COVID-19 virus mainly spreads through close contact\footnote{\url{https://www.cdc.gov/coronavirus/2019-ncov/prevent-getting-sick/how-covid-spreads.html}}. 
Fortunately, the disease transmission probability can be decreased by the efforts of the individual. For example, an individual can choose to avoid public and closed spaces, wear a protective mask or start shopping online. When two people meet the disease transmission probability depends on both sides’ effort. The disease is less likely to occur if both parts are wearing protective masks instead of just one part. However, the decrease in the risk of transmission is not additive in the interacting agents' efforts. Using this intuition, we assume in this paper that disease transmission likelihood depends on the efforts of people in a multiplicative way. The effort of the agents to minimize their risk will therefore be given the name ``contact factor control", in line with the game-based model introduced in \cite{aurell2020optimal}.

Before giving the proper description of the epidemiological graphon game model, how it can be approached numerically, and all technical details, we enlarge upon the heuristics of graphon-based interaction with contact factor control and review the literature of related fields of research in the next sections.

\subsection{The SIR Graphon Game with Contact Factor Control}
\label{sec:intro-example}

The most famous compartmental model in epidemics is arguably the classical Susceptible-Infected-Removed (SIR) model\footnote{Depending on the modeling, R can refer to ``Removed" or ``Recovered".}. In this section, we take advantage of its wide familiarity and compact formulation to further motivate for the concept of contact factor control and graphon-type interaction. 

In order to give the description of the rate at which agents become infected, we need to first introduce some notation. Consider $N$ individuals, each of whom transitions between the states Susceptible ($\mathsf{S}$), Infected ($\mathsf{I}$), and Removed ($\mathsf{R}$). An individual in state $\mathsf{R}$ has either gained immunity or deceased. Denote the state of agent $j\in\{1,\dots, N\}$ at time $t$ by $X^{j,N}_t$. A susceptible individual might encounter infected individuals, resulting in disease transmission. Encounters occur pairwise and randomly throughout the population with intensity $\beta$. The number of encounters with infected agents in a short time interval $[t-\Delta t,t)$ is approximately proportional to the the share of the population in state $\mathsf{I}$ at $t$. Between each agent pair $(j,k)$, we set the interaction strength to $w(x_j,x_k)$, where $w$ is a graphon (see Definition 1 for its definition) and $x_j, x_k$ are random variables uniformly distributed on $[0,1]$. 
Hence agent $k$'s transitions rate from state $\mathsf{S}$ to $\mathsf{I}$ is scaled by $w$, thus of the form $\beta \frac{1}{N}\sum_{k=1}^N w(x_j,x_k)\mathbf{1}_{\mathsf{I}}(X^{k,N}_{t-})$. Upon infection, an individual starts the path to recovery. The jump from state $\mathsf{I}$ to $\mathsf{R}$ happens after an exponentially distribution time with rate $\gamma$. The state $\mathsf{R}$ is absorbing. 
    
Denoting $Z^{j,N}_t := \frac{1}{N}\sum_{k=1}^N w(x_j,x_k)\mathbf{1}_{\mathsf{I}}(X^{k,N}_{t-})$, the transition rate matrix for player $j\in\{1,\dots, N\}$ is
\begin{equation}
    Q(Z^{j,N}_t) = 
    \begin{bmatrix}
        -\beta Z^{j,N}_t & \beta Z^{j,N}_t & 0
        \\
        0 & -\gamma & \gamma
        \\
        0 & 0 & 0
    \end{bmatrix}.
\end{equation}
At this stage, the distinguishability of the players is seen in the aggregate variables $(Z^{j,N}_t)_{j=1}^N$ which differ in value from player to player. As $N\rightarrow \infty$, we expect the probability distribution flow of player $j$ to converge to the solution of the ODE
\begin{equation}
\label{eq:SIR1}
    \dot{p}^{x_j}(t) = p^{x_j}(t)Q(Z^{x_j}_t),\quad p^{x_j}(0) = p_0^{x_j},
\end{equation}
where $p_0^{x_j}$ is some given initial distribution over the states and
\begin{equation}
    Z^{x_j} := \int_I w(x_j,y)p^y_t(\mathsf{I})dy.
\end{equation}
Scaling $p^x_t$, $x\in [0,1]$, by a population size $N$, $Np^x_t =: (S^x(t), I^x(t), R^x(t))$,
we retrieve a formulation of the compartmental SIR model with graphon-based interactions 
\begin{align*}
    \dot{S}^x(t) &= -\frac{\beta}{N}\left(\int_I w(x,y)I^y(t)dy\right)S^x(t),& S^x(0) &= Np^x_0(\mathsf{S}),
    \\
    \dot{I}^x(t) &= \frac{\beta}{N}\left(\int_Iw(x,y)I^y(t)dy\right)S^x(t) - \gamma I^x(t),& I^x(0) &= Np^x_0(\mathsf{I}),
    \\
    \dot{R}^x(t) &= \gamma I^x(t),& R^x(0) &= Np^x_0(\mathsf{R}).
\end{align*}

So far we have considered a model without control. Assuming agents choose their ``contact factor" in line with the discussion above in order to decrease their risk of getting infected, we can express this new feature mathematically as follows: Given that the meeting frequency is $\beta$, pairing is random, disease spreads from infected agents to susceptible, and the spread probability is scaled by the efforts of the individuals that meet in a multiplicative way, the transition rate for individual $j$ from $\mathsf{S}$ to the $\mathsf{I}$ becomes
\begin{equation}
\label{eq:intro-example-1}
    \beta \alpha^j_t \frac{1}{N}\sum_{k=1}^N w(x_j,x_k) \alpha^k_t \mathbf{1}_{\mathsf{I}}(X^{k,N}_{t-}),
\end{equation}
where $\alpha^k_t$ denotes the (contact factor) action of individual $k\in\{1,\dots, N\}$ at time $t$, selected from set of actions $A$. Along the lines of the heuristics of mean field game theory, we anticipate that in an appropriate approximation of our interacting system in the limit $N\rightarrow \infty$ the representative agent $x$ transitions from susceptible to infected with rate
\begin{equation}
\label{eq:intro-example-2}
    \beta \alpha_{t} \int_Iw(x,y)\left(\int_A a \rho^y_t(da,\mathsf{I})\right)dy,
\end{equation}
where $\rho^y_t$ is the joint distribution of action and state of player $y$ at time $t$. 

\subsection{Related Literature}

\subsubsection{Graphon Games and Finite State Space Mean Field Games} 
\label{subsubsec:mfg_lit}
Challenges related to the large number of agents arise in the game theory, and mean field game (MFG) theory presents a toolbox to compute equilibria in these games with large number of players. MFGs were first developed for continuous state space models \cite{lasry2006jeux,lasry2006jeux2,huang2006large} and later for a finite state space models \cite{gomes2010discrete,kolokoltsov2012nonlinear,Gomes2013}. The theory for finite state MFGs has been extended in many directions, with contributions including a probabilistic approach \cite{Cecchin2018}, the master equation approach~\cite{bayraktar2018analysis}, minor-major player games \cite{Carmona2016} and extended games\footnote{In the MFG setup, ``extended" refers to the models that the interactions are happening through the joint distribution of action and state of players.} \cite{Carmona2018Extended}. Further, finite state mean field control, risk-sensitive control, and zero-sum games are treated in \cite{choutri2018stochastic,choutri2019mean,choutri2019optimal} which cover cases of unbounded jump intensities. Graphon games \cite{delarue2017,Caines_2018,caines2019graphon,Parise_2019,carmona2019stochastic,gao2020lqg,aurell2021stochastic}
have recently been receiving an increasing research interest. The motivation is the study of strategic decision making in the face of a large non-complete but dense networks of distinguishable agents. The graphon game's rising popularity stems from its ability to handle heterogeneity of agents to an extent far beyond the MFG theory.

\subsubsection{Mean Field Games and Related Models for Epidemics}

Decision making in compartmental models (\textit{e.g.}, the SIR model) have been studied intensively for a long time, with an increasing interest recently with the COVID-19 pandemic. In the form of games and optimal control problems, disease-combating efforts ranging from strategies for social contracts to vaccination have been analyzed in the literature. Here, we focus on work relying on the graphon game theory and the mean-field approach.

During an epidemic where the disease is prone to close-contact transmission (one example being COVID-19) control of the contact rate or other social distancing protocols are go-to solutions in the fight against the disease. Such strategies and related variations have been studied in the context of mean field game \cite{Elie2020,Cho2020}, mean field optimal control \cite{lee2020}, and mean field type games \cite{Tembine2020}. In order to accurately asses the risks in the decision making process, in reality the population needs to be tested for the disease. Two recent papers studying optimal testing policies are \cite{charpentier2020covid,Hubert2020}. The effects of vaccination and related decision makings are not studied in this paper, but they have been analyzed in MFG-related settings since before the COVID-19 pandemic \cite{laguzet2015individual,laguzet2016equilibrium,salvarani2016individual,hubert2018nash,doncel2017mean,gaujal20vaccination}. Recently, the interplay between a population in which a disease spreads and a regulator has been studied in the form of Stackelberg games. In such models, the members in the populations are taking actions based on policies issued by the regulator, while the regulator anticipates the population's reaction and optimizes the policy. The case of a cooperative population has been studied in \cite{Hubert2020}, while in \cite{aurell2020optimal}, a population of selfish agents with contact factor control has been studied.

An example of a deterministic optimal control problem for centralized decision making during a pandemic in a society with multiple communicating subpopulations is given in \cite{Gao_2019}. The subpopulations interact over a non-uniform graph. A central planner wants to flatten the (global) curve of infections, leading to the optimal control problem. Sending the number of subpopulations to infinity, we anticipate a limit where each interaction is weighted by a graphon and the limit model would be reminiscent of the interacting system of Kolmogorov equations studied in this paper.

\subsubsection{Networks and Graphons in Epidemiology}

There is a vast body of literature on epidemiology modeling with network interactions. A review of the studies that use idealized networks\footnote{Some examples of idealized networks given in this work are Small-World Networks, Scale-Free Networks, Exponential Random Graph Models.} in epidemiology models can be found in \cite{Keeling_2005}. More closely related to the ideas in this paper, there are recent contributions connecting epidemic models and graphons. In \cite{vizuete2020graphonbased} a sensitivity analysis on the graphon SIS epidemic model is conducted. An infinite dimensional SIS model with application to targeted vaccination is considered in \cite{delmas2021targeted}. The paper \cite{keliger2020localdensity} proposes a model with local-density dependent Markov processes interacting through a graphon structure, and they consider applications to epidemiology. In a similar but more general setting to the SIR model with graphon interaction, \cite{andersson1995limit} studies convergence of a stochastic particle system to an SIR-like system of PDEs with spatial interaction. We note \cite{andersson1995limit} and its continuation \cite{djehiche1995rate,djehiche1998large} may be relevant for a future study of the convergence of $N$-player Nash equilibria to the equilibria in the finite state graphon game. The works mentioned only consider the dynamics of the population without taking the agents' decision making into account.

\subsection{Contributions and Paper Structure}

This paper contributes to the analysis and numerical resolution of finite state graphon games, with applications to epidemiology model. We construct a probabilistic particle model for a continuum of interacting agents and we prove that graphon aggregates must be deterministic (as in \textit{e.g.} \eqref{eq:intro-example-2}) under a set of conditions on the strategies and transition rates. This motivates the study of the asymptotic deterministic model formulation while giving us a transparent interpretation of the agent's control in the applied context. We derive theoretical results about the deterministic model: a verification theorem and an existence theorem for the coupled continuum of forward-backward ordinary differential equations (FBODEs) that characterizes the finite state graphon game at equilibrium are given. We also propose a machine learning method to solve the FBODE system. Finally, we consider an graphon game model for epidemic disease spread. Multiple test cases are solved with the proposed numerical method and the experimental results are discussed.

The outline of the rest of the paper is as follows: In Section~\ref{sec:model}, we introduce the model and analyze its deterministic formulation. In Section~\ref{sec:numerical}, we introduce the numerical approach and give experiment results. In Section~\ref{sec:fubini}, a theoretical framework for the model's probabilistic framework is presented and we rigorously define the graphon game. For the sake of conciseness, the proofs are passed to the appendices.

\section{Model}
\label{sec:model}

\subsection{Setup and Preliminaries}

Let $n\in \mathbb{N}$ and let $E$ be the finite set $\{1,\dots, n\}$. For each $e\in E$ define the difference operator $\Delta_e$ acting on functions on $E$ by the formula $[\Delta_e \phi](e^{\prime}) = \phi(e^{\prime}) - \phi(e)$. We identify the set of probability measures on $E$, $\mathcal{P}(E)$, with the simplex $\Delta(E) := \{x = (x_1,\dots, x_n)\in \mathbb{R}^n_+: \sum_i x_i = 1\}$ and endow it with the Euclidean distance. Throughout the paper, the notation $\mathcal{P}(\cdot)$ will be used to denote the set of Borel probability measures.

Let $T>0$ be a finite time horizon. A process $(f_t)_{t\in[0,T]}$ will be denoted with its bold letter symbol $\boldsymbol f$. Let $\mathcal{C} := C([0,T]; \mathbb{R})$ be the space of continuous real-valued functions from $[0,T]$, $\mathcal{D} := D([0,T]; \mathbb{R})$ be the space of real-valued functions from $[0,T]$ c\`adl\`ag at $t\in [0,T)$ and continuous at $t=T$. We denote the uniform norm by $\|x\|_T := \sup_{s\in[0,T]}|x(s)|$, $x\in \mathcal{D}$. We note that $(\mathcal{C}, \|\cdot\|_T)$ and $(\mathcal{D}, \|\cdot\|_T)$ are both Banach spaces, only the former is separable. Let $\mathcal{D}_E \subset \mathcal{D}$ be the set of functions in $f\in \mathcal{D}$  such that $f([0,T])\subset E$. Since $\mathcal{D}_E$ is a closed subset of $\mathcal{D}$, $(\mathcal{D}_E, \|\cdot\|_T)$ is a complete metric space.

Let $I$ be the unit interval equipped with the Euclidean distance. We denote by $\lambda_I$ and $\mathcal{B}(I)$ the Lebesgue measure and Borel $\sigma$-field on $I$, respectively. The set $I$ is indexing the continuum of players in the graphon game. Troughout the paper, we will employ the notation $\underline \phi := (\phi(x))_{x\in I}$ for functions with domain $I$. Furthermore, in most cases we will denote the index argument with a superscript:
$\phi^x := \phi(x)$. 

This paper studies games with heterogeneous interactions. When the players in the game interact, the weight they give to each others' action is parameterized by their indices. This weighted averaging is captured as the integration with respect to a graphon kernel function (see Section 1 for the discussion). Here, we give the most prevalent definition of a graphon.
\begin{definition}
\label{def:graphon}
    A graphon is a symmetric Borel-measurable function, $w : I\times I \rightarrow [0,1]$. 
\end{definition}
The graphon induces an operator $W$ from $L^2(I)$ to itself: for any $\underline\phi\in L^2(I)$,
\begin{equation}
    \label{eq:graphon-operator}
    [W\underline\phi]^x := \int_I w(x,y)\phi^ydy.
\end{equation}
\begin{remark}
The definition of a graphon varies somewhat in the literature. Some authors require neither symmetry nor a non-negative range. As mentioned in the introduction, the graphon defined as here can be used to represent the limit of a sequence of dense
random graphs as the number of nodes (players) goes to infinity. For example, the constant graphon $w(x,y) = p$ is in a sense the limit of a sequence of Erd\H{o}s-R\'enyi graphs with parameter $p \in [0,1]$. Conversely, random graphs can be sampled from a graphon in at least two different ways: either we sample points in $I$ and construct a weighted graph whose weights are given by the graphon, or we sample points in $I$ and then sample edges with probabilities given by the graphon. 
\end{remark}

A $Q$-matrix with real-valued entries $q_{i,j}$, $i,j\in E$, is an $n\times n$ matrix with non-negative off-diagonal entries such that:
\begin{equation}
    q_{i,i} = -\sum_{j=1: j\neq i}^n q_{i,j},\quad i\in E.
\end{equation}
In this paper, we will consider controlled $Q$-matrices with entries that depend on population aggregates. More specifically, we let for each $x\in I$, $q^x_{i,j}: A \times \mathbb{R} \mapsto \mathbb{R}$, $i,j\in E$, be bounded measurable functions such that $Q^x(\alpha, z) := [q^x_{i,j}(\alpha,z)]_{i,j=1}^n$ is a $Q$-matrix for all $(\alpha,z) \in A\times \mathbb{R}$. We are going to work under the following assumption on the rates $q^x_{i,j}$:
\begin{condition}
\label{eq:cond_on_q}
\begin{enumerate}[label={(\roman*)}]
    \item\label{eq:cond_on_q-bdd}  There is a finite constant $q_{\max}>0$ such that for all $x\in I$, $(i,j)\in E^2$, $a \in A$, and $z\in \mathbb{R}$:
    \begin{equation}
        |q^x_{i,j}(a,z)| \leq q_{\max}.
    \end{equation}
    \item\label{eq:cond_on_q-lip} 
    There is a finite constant $C>0$, possibly depending on $n$, such that for $p = 1,2$ and for all $x\in I$, $t\in[0,T]$, $\alpha\in A$, $(i,j)\in E^2$, and $z,z' \in \mathbb{R}$
    \begin{align*}
    &
    |q^x_{i, k+i}(\alpha, z) - q^x_{j, k+j}(\alpha, z')|
    \leq
    C\left(\mathbf{1}_{\{i\neq j\}} + |z - z'|^p\right).    
\end{align*}
\end{enumerate}
\end{condition}

Since the state space is finite and the rates are assumed to be uniformly bounded in Condition~\ref{eq:cond_on_q}.(i), the H\"older-continuity assumption of Condition~\ref{eq:cond_on_q}.(ii) is less restrictive than might appear at first glance.

\subsection{The Finite State Graphon Games Model for Epidemiology}

In this section, we give a descriptive introduction to the epidemiological graphon game without going in to all the technical details. A rigorous mathematical motivation, built on the theory of Fubini extensions to accommodate for a continuum of independent jump processes, is presented in Section~\ref{sec:fubini}.

On a probability space $(\Omega, \mathcal{F}, \mathbb{P})$, we consider a continuum of $E$-valued pure jump processes $\boldsymbol X^x = (X^x_t)_{t\in [0,T]}$ indexed over $x\in I$. That is, for each $\omega\in \Omega$, $\boldsymbol X^x(\omega)\in \mathcal{D}_E$. The stochastic process $\boldsymbol X^x$ models the state trajectory of player $x$. The initial state $X^x_0$ is sampled from a distribution $p_0^x \in \mathcal{P}(E)$.
Each player implements a strategy $\boldsymbol\alpha^x$, a process taking values in the compact interval $A\subset \mathbb{R}$ described in more detail below. The players interact and each player's state trajectory is potentially influenced by the whole strategy profile $(\boldsymbol\alpha^x)_{x\in I}$. Therefore, in order to emphasize this dependence, we denote the state trajectory of player $x$ as $\boldsymbol X^{\underline{\boldsymbol\alpha},x}$ given a strategy profile $\underline{\boldsymbol\alpha}=(\boldsymbol\alpha^x)_{x\in I}$. For all $x\in I$, $\boldsymbol X^{\underline{\boldsymbol\alpha},x}$ is a $E$-valued pure jump process with rate matrix $Q^x(\alpha^x_t, Z^{\underline{\boldsymbol\alpha},x}_t)$ at time $t\in[0,T]$. The rate matrix is controlled by $\boldsymbol\alpha^x$ and influenced by the population aggregate $\boldsymbol Z^{\underline{\boldsymbol \alpha},x}$. The aggregates we consider are averages weighted by a graphon $w$, more specifically of the form $[W K(\underline\alpha_t, \underline X^{\underline{\boldsymbol\alpha}}_{t-})]$ (cf.  \eqref{eq:graphon-operator}) for some function $K : A\times E \rightarrow \mathbb{R}$. In Section~\ref{sec:fubini} we prove that the aggregate is a deterministic function of time, henceforth we write
\begin{equation}
\label{eq:det-aggregate}
    Z^{\underline{\alpha},x}_t = \int_I w(x,y)\mathbb{E}\left[K(\alpha^y_t, X^{\underline{\alpha},y}_{t-})\right]dy,
\end{equation}
where $w$ is a graphon. 

$K$ will be called the \textit{impact function} since it quantifies how much a player's state or control impacts the aggregate variable. One example is the impact function in the Section~\ref{sec:intro-example} where $K(\alpha,e) = \alpha\mathbf{1}_{(e = \mathsf I)}$, where the interpretation being that the aggregate is the averaged contact factor control of infected players. We have the following assumption on $K$:
\begin{condition}
\label{cond:K_lip}
There exist finite constants $L_K,C_K>0$ such that for all $a,a'\in A$ and $e\in E$:
\begin{equation*}
    |K(a,e) - K(a',e)| \leq L_K |a-a'|,\quad |K(a,e)| \leq C_K.
\end{equation*}
\end{condition}

Our (for now) formal definition of the probabilistic interacting system of players is complete. The rigorous analysis of the system, the construction of a continuum of state trajectories, and conditions under which the aggregate is deterministic, \textit{i.e.}, of the form \eqref{eq:det-aggregate}, is treated in detail in Section~\ref{sec:fubini}. 

The key feature of the graphon game is that the aggregate variable is in general not the same for two distinct players. The players are therefore distinguishable and there is no ``representative agent", as in MFGs\footnote{This is of course not the case if the graphon is constant. In this case, the graphon game is in fact equivalent to an MFG, so there is a representative agent. There are a few more examples of this kind, such as the piecewise constant graphon yielding a game equivalent to a multipopulation MFG with one representative agent for each subpopulation.}. As a direct consequence, there is no flow of player state distributions common to all players. Instead, each player has their own flow. Denote by $p^{\underline{\boldsymbol\alpha},x}(t,e)$  the probability that player $x\in I$ that is in state $e\in I$ at time $t\in [0,T]$, given that the population plays the strategy profile $\underline{\boldsymbol\alpha}$. We shall argue that player $x$'s state distribution flow $\boldsymbol p^{\underline{\boldsymbol\alpha},x}$ solves the Kolmogorov forward equation
\begin{align}
\label{eq:dynamics}
    &\frac{d}{dt}p^{{\underline{\boldsymbol\alpha},x}}(t)
    =
    p^{\underline{\boldsymbol\alpha},x}(t)
    Q^x(\alpha^x_t, Z^{\underline{\boldsymbol\alpha},x}_t),\quad p^{\underline{\boldsymbol\alpha},x}(t) := (p^{\underline{\boldsymbol\alpha},x}(t,e))_{e\in E}
\end{align}
with initial condition $p^{\underline{\boldsymbol\alpha},x}(0) = p_0^x$ and the player's aggregate variable $Z_t^{\underline{\alpha},x}$ is
\begin{equation}
\label{eq:aggregage-Z-alphabar}
    Z^{\underline{\boldsymbol\alpha},x}_t
    =
    \int_I w(x,y) \left(\int_{A\times E} K(\alpha,e) \rho^{\underline{\boldsymbol\alpha},y}_t(d\alpha,de)\right) dy,
\end{equation}
with $\rho^{\underline{\boldsymbol\alpha},y}_t$ being the joint probability law of control and state, $(\alpha^y_t, X^{\underline{\boldsymbol\alpha},y}_{t-})$. 

We turn our focus to the players' actions. We will make three standing assumptions that directly affect what strategies players will choose. The first is that the environment that endogenously 
affects the players (but is known to the players) is varying smoothly over time, with no abrupt changes for example in lockdown penalties or expected recovery time. Secondly, if the players can affect their environment with their control, then the environment varies smoothly with their control too. For example, the risk of infection depends continuously on the agent's level of cautiousness. Finally, the players' strategies are decentralized, \textit{i.e.}, may not be affected by the transition of any agent other than themselves. Under these circumstances, the players have no apparent reason to discontinuously change their action over time except at times of transition between states. Such strategies ($A$-valued; decentralized; continuous in time between changes of the player's own state) will be called admissible and the set of admissible strategies denoted by $\mathbb{A}$. The setting is further discussed in Section~\ref{sec:fubini}.

In this paper, we consider the finite horizon problem where the cost is composed of two components: a running cost and a terminal cost. For each player $x\in I$, the assumptions on the conditions that the running and terminal cost functions,
$f^x : [0,T]\times E\times \mathbb{R}\times A \rightarrow \mathbb{R}$ and $g^x: E\times\mathbb{R}\rightarrow\mathbb{R}$, satisfy are given later in the text together with the theoretical results. The total expected cost to player $x$ for playing the strategy $\boldsymbol\sigma\in \mathbb{A}$ while the population plays the strategy profile $\underline{\boldsymbol \alpha}$ is
\begin{align}
\label{eq:costfunction}
    &\mathcal{J}^x(\boldsymbol\sigma; \underline{\boldsymbol\alpha}) 
    =
    \mathbb{E}\left[\int_0^Tf^x(t, X^{\underline{\boldsymbol\alpha},x}_t, Z^{\underline{\boldsymbol\alpha},x}_t, \sigma_t)dt + g^x(X^{\underline{\boldsymbol\alpha},x}_T, Z^{\underline{\boldsymbol\alpha},x}_T)\right].
\end{align}
As we shall see, a change in player $x$'s control has no effect on the aggregate. Hence, the expected cost depends on the strategy profile only indirectly through the value of the aggregate variable. See Section~\ref{sec:fubini} for the details. Therefore, hereinafter we shall use the notation $J^x(\boldsymbol\sigma; \boldsymbol Z^{\underline{\boldsymbol\alpha},x}) $ for the above right hand side. In light of this, we employ the following definition of a Nash equilibrium in the graphon game:
\begin{definition}
\label{def:def-GGNE}
The strategy profile $\underline{\boldsymbol{\alpha}}$ is a Nash equilibrium if it is admissible and no player can gain from a unilateral deviation, \textit{i.e.}, 
\begin{equation*}
    J^x(\boldsymbol \alpha^x; \boldsymbol{Z}^{\underline{\boldsymbol\alpha},x}) 
    \leq 
    J^x(\boldsymbol\sigma; \boldsymbol{Z}^{\underline{\boldsymbol\alpha},x}),
    \quad \forall x\in I,\ \forall \boldsymbol\sigma\in\mathbb{A}.
\end{equation*}
\end{definition}

\subsection{Analysis of Finite State Graphon Games}

By Definition~\ref{def:def-GGNE}, an admissible strategy profile $\hat{\underline{\boldsymbol{\alpha}}}$ is a Nash equilibrium if there exists an aggregate profile $\hat{\underline{\boldsymbol{Z}}} = (\hat{\boldsymbol Z}^x)_{x\in I}$ such that
\begin{itemize}
    \item for all $x \in I$, $\hat{\boldsymbol{\alpha}}^x$ minimizes $J^x(\cdot; \hat{\boldsymbol{Z}}^x)$;
    \item for all $x \in I$, $\hat{\boldsymbol{Z}}^x$ is the aggregate perceived by player $x$ if the population uses strategy profile $\hat{\underline{\boldsymbol{\alpha}}}$.
\end{itemize}
This alternative formulation has the advantage to split the characterization of the equilibrium into two parts and in the first part, the optimization problem faced by a single agent is performed with while the aggregate is fixed.

With a flow $\boldsymbol{Z}^x$ being fixed, we define the value function of player $x\in I$ as, for $t\in [0,T]$ and $e\in E$,
\begin{align*}
    &u^x(t,e) := 
    \\
    &
    \inf_{\boldsymbol\sigma\in \mathbb{A}}\mathbb{E}\left[\int_t^T f^x(s,X^{\boldsymbol\sigma,\boldsymbol{Z}^x,x}_s, Z^x_s, \sigma_s)dt + g^x(X^{\boldsymbol\sigma,\boldsymbol{Z}^x,x}_T, Z^x_T)\ |\ X^{\boldsymbol\sigma,\boldsymbol{Z}^x,x}_t = e \right]
\end{align*}
where $\boldsymbol X^{\boldsymbol\sigma,\boldsymbol{Z}^x,x}$ is an $E$-valued pure jump process with transition rate matrix $Q^x(\sigma_t, Z^x_t)$ at time $t\in[0,T]$ and initial distribution $p^x_0$. 

To derive optimality conditions, we introduce the Hamiltonian of player $x$:
\begin{equation}
    H^x(t,e, z, h, a) := \overrightarrow{\mathbf{1}_e} Q^x(a,z)h + f^x(t,e,z,a),
\end{equation}
where $\overrightarrow{\mathbf{1}_e}$ is the coordinate (row) vector in direction $e$ in $\mathbb{R}^n$ and $h\in \mathbb{R}^{n}$. We assume that $A \ni \alpha \mapsto H^x(t,e,z,h,\alpha)$ admits a unique measurable minimizer $\hat a^x_e(t,z,h)$ for all $(t,z,h)$ and define the minimized Hamiltonian of player $x$:
\begin{equation}
\label{eq:-gg-Hopt}
    \hat{H}^x(t, e, z, h) := H^x(t,e,z,h,\hat{a}^x_e(t,z,h)).
\end{equation}
The dynamic programming principle of optimal control leads to the HJB equation for $u^x$ that reads 
\begin{equation}
\label{eq:gg-HJB}
    \dot{u}^x(t,e) + \hat{H}^x(t,e,Z^x_t,  u^x(t,\cdot)) = 0, \quad u^x(T,e) = g^x(e, Z^x_T),
\end{equation}
where $\dot{u}^x(t,e)$ denotes the time derivative of ${u}^x(t,e)$. Noting that $\overrightarrow{\mathbf{1}_e} Q(a,z)h = \overrightarrow{\mathbf{1}_e} Q(a,z) \Delta_e h$, \eqref{eq:gg-HJB} can be equivalently written as
\begin{equation}
\label{eq:gg-HJB-Delta}
    \dot{u}^x(t,e) + \hat{H}^x(t,e,Z^x_t, \Delta_e u^x(t,\cdot)) = 0, \quad u^x(T,e) = g^x(e, Z^x_T).
\end{equation}
In the following theorem, we verify that the solution of the HJB equation indeed gives the value function of the infinitesimal agent's control problem, and we provide an expression for an optimal Markovian control in terms of this value function and the aggregate.
\begin{theorem}
\label{theorem:feedback-control}
If $u^x : [0,T]\times E \ni (t,e) \mapsto u^x(t,e)\in \mathbb{R}$ is a continuously differentiable solution to the HJB equation~\eqref{eq:gg-HJB}, then $u^x$ is the value function of the optimal control problem when the flow $\boldsymbol Z^x$ is given. Moreover, the function
\begin{equation}
    \hat{\phi}^x(t,e) = \hat{a}_e^x(t,Z^x_t, u^x(t,\cdot))
\end{equation}
gives an optimal Markovian control.
\end{theorem}    

Next, we prove the existence of a solution to the coupled Kolmogorov-HJB system at equilibrium. For that purpose we place the following condition:
\begin{condition}
\label{cond:cond_differentiability}
\begin{enumerate}[label={(\roman*)}]
    \item \label{cond:diff_q} 
    There exist two functions $Q_1$ and $Q_2$ with $Q_2$ locally Lipschitz such that $Q(a,z) = Q_1(z) + aQ_2(z)$ for all $a \in A, z \in \mathbb{R}$.
    \item \label{cond:diff_f} 
    $a \mapsto f^x(t,e,z,a)$ is continuously differentiable, and strongly convex uniformly in $(t,e,z)$ with constant $\lambda$; $(t,z) \mapsto \partial_a f^x(t,e,z,a)$ is locally Lipschitz continuous, uniformly in $a \in A, e \in E$.  
    \item \label{cond:bound_f_g} $f$ and $g$ are uniformly bounded and $z \mapsto f^x(t,e,z,a)$ is continuous.
\end{enumerate}
\end{condition}    
    
As a consequence of Condition~\ref{cond:cond_differentiability}.\ref{cond:diff_q} and~\ref{cond:cond_differentiability}.\ref{cond:diff_f} $a \mapsto H^x(t,e, z, h, a)$ is once continuously differentiable and strictly convex, 
and $(t,z,h) \mapsto \hat a_e(t,z,h)$ is locally Lipschitz continuous (see Lemma~\ref{lem:reg-hat-a} in Appendix~\ref{app:lemma1}). We denote Lipschitz constant of $[-c,c]\ni z \mapsto \hat{a}(t,z,h)$ by $L_{\hat a}(c)$, which can be bounded from above using smoothness properties of $Q^x$ and $f^x$. More specifically, $L_{\hat a}$ depends on the local Lipschitz coefficients of $z \mapsto \partial_a Q^x(a,z)$ and $(t,z,h) \mapsto \partial_a f^x(t,e,z,a)$, see the proof of Lemma~\ref{lem:reg-hat-a}).
Recall that $C_K$ denotes the uniform upper bound of the impact function $K$ guaranteed by Condition~\ref{cond:K_lip}.
\begin{theorem}
\label{theorem:existence-ode}    
    Assume Condition~\ref{eq:cond_on_q}, \ref{cond:K_lip}, \ref{cond:cond_differentiability} hold. 
    If $\|w\|_{L^2(I\times I)}L_KL_{\hat a}(C_K) < 1$,
    then the coupled HJB--Kolmogorov forward-backward system at equilibrium 
\begin{equation}
\label{eq:FB-hat-phi}
    \left\{
    \begin{aligned}
    &\dot p^x(t, e) = \sum_{e^{\prime} \in E} q^x_{e^{\prime},e}(\hat \phi^x(t, e^{\prime}), Z_t^x) p^x(t, e^{\prime}), \quad \forall e \in E,\ t\in [0,T]\\
    &\dot u^x(t, e) = -\hat H^x(t, e, Z_t^x, \Delta_e u^x(t, \cdot)), \quad \forall e \in E,\ t \in [0,T]\\
    &Z_t^x  = \int_I w(x,y) \left(\sum_{e\in E}K(\hat \phi^y(t,e), e)p^y(t,e)\right)\lambda(dy)
    \\
    &u^x(T,e) = g^x(e, Z^x_T),\quad p^x(0,e) = p^x_0(e), \\
    &\hat{\phi}^x(t,e) = \hat{a}_e^x(t,Z^x_t, u^x(t,\cdot))
    \end{aligned}
    \right.
\end{equation}
admits a bounded solution $(u,p)$ in $C([0,T]; L^2(I\times E)\times L^2(I\times E))$ such that for each $(t,x)\in [0,T]\times I$, $p^x(t, \cdot)$ is a probability mass function on $E$.

\end{theorem}

\subsection{The Finite State Graphon Game for SIR with Contact Factor Control}
\label{sec:example1}

Here, we introduce a model which we shall use as a test bed for the numerical algorithm presented in Section~\ref{sec:numerical}.
It is inspired by the first example scenario in \cite{aurell2020optimal} and builds on the case discussed in Section~\ref{sec:intro-example}. It is a compartmental model with four possible states: $(\mathsf S)$usceptible, $(\mathsf I)$nfected, $(\mathsf R)$ecovered and $(\mathsf D)$eceased. The agents choose their level of contact factor. A regulator (government or health care authority) recommends state-dependent contact factor levels to the agents, denoted by $\boldsymbol\lambda^{(e)}$, $e\in \{\mathsf{S}, \mathsf{I}, \mathsf{R}\}$. For enforcement purposes, it also sets penalties for deviation from these levels. The cost has 3 components: The first component penalizes the agent for not following the regulator's recommended contact factor level, the second one is the cost of treatment for an infected agent (this cost can be player specific due to individual differences in health care plan coverage, etc.),
and the last one is the cost for being deceased.
In this setting, the running cost is written as
\begin{equation}
\begin{aligned}
\label{eq:SIR_cost}
    f^x(t,e,z,\alpha) 
    =
    \frac{c_\lambda}{2}\left(\lambda^{(\mathsf S)}(t) - \alpha\right)^2 \mathbf{1}_{(e = \mathsf S)}
    + \frac{1}{2}\left(\lambda^{(\mathsf R)}(t)-\alpha\right)^2 \mathbf 1_{(e = \mathsf R)}&
    \\
    + \left(\frac{1}{2}\left(\lambda^{(\mathsf I)}(t) - \alpha\right)^2 + c_I(x)\right)\mathbf 1_{(e = \mathsf I)} 
    + c_D(x) \mathbf 1_{(e = \mathsf D)},&
\end{aligned}
\end{equation}
where $c_I$ and $c_D$ are nonnegative costs functions (of the player index).
We set the terminal cost to be identically zero, $g^x(e,z) = 0$ for all $(e,z)\in E \times \mathbb{R}$. The transition rate matrix for player $x$ is given as: 
\begin{equation}
\label{eq:SIR-Qmatrix}
    Q^x(\alpha,z)
    =
\begin{blockarray}{ccccc}
    \begin{block}{ccccc}
      & \mathsf S & \mathsf I & \mathsf R & \mathsf D \\
     \end{block}
    \begin{block}{c[cccc]}
    &&&&\\[-1.5mm]
    \mathsf S &\cdots & \beta(x)\alpha z & 0 & 0
    \\
    \mathsf I &0 & \cdots & \rho(x)\gamma(x) 
    & (1-\rho(x))\gamma(x)
    \\
    \mathsf R &\kappa(x)  & 0 & \cdots 
    & 0
    \\
    \mathsf D & 0 & 0 & 0 & \cdots\\
    &&&&\\[-2mm]
    \end{block}
\end{blockarray},
\end{equation}
where $\beta, \gamma, \kappa,\rho$ are nonnegative parameter functions, $0\leq \rho\leq 1$, determining the rates of infection, recovery, reinfection, and decease, and where in each row, the diagonal entry is the negative of the sum of the other terms on the same row. In line with the discussion in Section~\ref{sec:intro-example}, the transition rate from state $\mathsf{S}$ to $\mathsf{I}$ depends on the agent's own decision and the aggregate variable\footnote{The rate is presented in \eqref{eq:SIR-Qmatrix} as $\beta(x)\alpha z$. For an arbitrary impact function $K$ this would lead to a violation of Condition~\ref{eq:cond_on_q}.\ref{eq:cond_on_q-bdd}. Here we are however considering contact factor control and aggregates of the form \eqref{eq:aggregage-Z-alphabar}, which are bounded by  Condition~\ref{cond:K_lip} and the compactness of $A$. The $Q$-matrix \eqref{eq:SIR-Qmatrix} can be modified without changing the model so that Condition~\ref{eq:cond_on_q}.\ref{eq:cond_on_q-bdd} is satisfied; we defer from this for the sake of presentation.}. Furthermore, when an infected agent transitions, she goes to state $\mathsf{R}$ with probability $\rho$ and to state $\mathsf{D}$ with probability $(1-\rho)$. 
Then, for player $x$ the optimality conditions yield
\begin{align*}
&\hat\phi^x(t,\mathsf S) = \lambda^{(\mathsf S)}(t) + \frac{\beta(x)}{c_\lambda}Z^x_t(u^x(t,\mathsf S) - u^x(t,\mathsf I))
\\
&\hat \phi^x(t,\mathsf I) = \lambda^{(\mathsf I)}(t)
\\
&\hat\phi^x(t,\mathsf R) = \lambda^{(\mathsf R)}(t)
\end{align*}
and the forward-backward graphon ODE system reads:
\begin{align*}
    &\dot p^x(t, \cdot) = p^x(t, \cdot)Q^x(\hat\phi^x(t, \mathsf{S}), Z^x_t)
    \\
    &\dot u^x(t,\mathsf S)= \beta(x)\hat\phi^x(t,\mathsf S)Z^x_t\big(u^x(t,\mathsf S) - u^x(t,\mathsf I)\big) -\frac{c_\lambda}{2} \big(\lambda^{(\mathsf S)}(t)-\hat \phi^x(t, \mathsf S)\big)^2,
    \\
    &\dot u^x(t,\mathsf I)
    = \rho(x)\gamma(x)\left(u^x(t,\mathsf I) - u^x(t,\mathsf R)\right) 
    + (1-\rho(x))\gamma(x)\left(u^x(t,\mathsf I) - u^x(t,\mathsf D)\right)- c_I(x),
    \\
    &\dot u^x(t,\mathsf R)= \kappa(x)\left(u^x(t,\mathsf R) - u^x(t,\mathsf S)\right),
    \\
    &\dot u^x(t,\mathsf D)= -c_D(x),
    \\
    & u^x(T,e)= 0,\quad p^x(0,e) = p^x_0(e),\quad e\in \{\mathsf{S}, \mathsf{I}, \mathsf{R}\},
    \\
    & Z^x_t= \int_I w(x,y)\hat \phi^y(t,\mathsf I)p^y(t,\mathsf I) dy,\quad t\in[0,T],\ x\in I.
\end{align*}

We note that with a careful choice $\beta(x)$, $c_\lambda$, and $(\boldsymbol \lambda^{(e)})_{e\in E}$ this system will satisfy the sufficient condition for existence from Theorem~\ref{theorem:existence-ode}.

\section{Numerical Approach}
\label{sec:numerical}
We rewrite the continuum of FBODEs \eqref{eq:FB-hat-phi} as the solution of a minimization problem: minimize
\begin{equation}
\label{eq:J-NN-initonly}
    \mathbb{J}(\theta) 
    = 
    \int_{I} \sum_{e \in E} |u^x_{\theta}(T,e) - g^x(e, Z_{\theta,T}^x)|^2 dx,
\end{equation}
where $(p_{\theta}, u_{\theta})$ solve the forward-forward continuous system of ODEs:
\begin{equation}
    \left\{
    \begin{aligned}
    &\dot p^x_{\theta}(t, e) 
    = 
    \sum_{e^{\prime} \in E} q^x_{e^{\prime},e}\big(Z_{\theta,t}^x, \hat a^x(t, e^{\prime}, Z^x_{\theta,t}, u^x_{\theta}(t,\cdot))\big) p^x_{\theta}(t, e^{\prime}), 
    \\
    &\dot u^x_{\theta}(t, e) = -\hat H^x(t, e, Z_{\theta,t}^x, \Delta_e u^x_{\theta}(t, \cdot)),
    \\
    &Z_{\theta,t}^x  = \int_I w(x,y) \left(\sum_{e \in E} K\big(\hat a^y(t, e, Z^y_{\theta,t}, u^y_{\theta}(t, \cdot)\big), e \big) p^{y}_{\theta}(t,e) \right)\lambda(dy) 
    \\
    &u^x_{\theta}(0,e) = \varphi^x_{\theta}(e), \quad p^x(0,e) = p^x_0(e), \quad e \in E,\ t\in [0,T]\ .
    \end{aligned}
    \right.
\end{equation}
This ``shooting'' strategy is reminiscent of the one used \textit{e.g.} in~\cite{MR1766421,MR2963805,MR3738664} for stochastic optimal control problem and \textit{e.g.} in~\cite{carmona2019convergenceFinite,aurell2020optimal} for mean field games in a numerical context. However, here we deal with a continuum of ODEs rather than a finite number of stochastic differential equations. Here $\theta$ is the (unknown) parameter in the function $\varphi$ replacing the initial condition of $u$. Typically, $\theta$ is a real-valued vector of dimension the number of degrees of freedom in the parametric function $\varphi$. See, \textit{e.g.},~\cite{carmona2021convergenceErgodic} for a description of the feedforward neural network architecture we used in the implementation.

Our strategy to find $\theta$ is to run a gradient-descent based method. To alleviate the computation cost and to introduce some randomness, at each iteration we replace the above cost $\mathbb{J}$ by an empirical average over a finite set of indices $x$, which is also used to approximate the value of the aggregate quantities. More precisely, for a finite set $\mathbf{S}$ of indices, we introduce
\begin{equation}
\label{eq:J-NN-Sample-initonly}
    \mathbb{J}^N(\theta,\mathbf{S}) 
    =  
    \frac{1}{N}\sum_{x \in \mathbf{S}} \sum_{e \in E} |u^x_{\theta,\mathbf{S}}(T,e) - g^x(e, Z_{\theta,\mathbf{S},T}^x)|^2,
\end{equation}
where $(p_{\theta,\mathbf{S}}, u_{\theta,\mathbf{S}})$ solves the forward-forward (finite) system of ODEs:
\begin{equation}
\label{eq:FFODE-finite-initonly}
    \left\{
    \begin{aligned}
    &\dot p^x_{\theta,\mathbf{S}}(t, e) 
    = 
    \sum_{e^{\prime} \in E} q^x_{e^{\prime},e}( Z_{\theta,\mathbf{S},t}^x, \hat a^x(t, e^{\prime}, Z^x_{\theta,t}, u^x_{\theta,\mathbf{S}}(t,\cdot))\big) p^x_{\theta,\mathbf{S}}(t, e^{\prime}), 
    \\
    &\dot u^x_{\theta,\mathbf{S}}(t, e) = -\hat H^x(t, e, Z_{\theta,\mathbf{S},t}^x, \Delta_e u^x_{\theta,\mathbf{S}}(t, \cdot)), 
    \\
    &Z_{\theta,\mathbf{S},t}^x  = \frac{1}{N}\sum_{y \in \mathbf{S}} w(x,y) \left(\sum_{e\in E} K\Big(\hat a^y(t, e, Z^y_{\theta,t}, u^y_{\theta,\mathbf{S}}(t,\cdot)\big),e\Big) p^{y}_{\theta,\mathbf{S}}(t,e) \right),
    \\
    &u^x_{\theta,\mathbf{S}}(0,e) = \varphi^x_{\theta}(e), \quad p^x(0,e) = p^x_0(e), \quad e \in E,\ x \in \mathbf{S}.
    \end{aligned}
    \right.
\end{equation}

\begin{algorithm}
\caption{DeepGraphonGame Algorithm\label{algo:DGG}}
\textbf{Input:} Initial parameter $\theta^{(0)}$; number of iterations $K$; sequence $(\beta_k)_{k = 0, \dots, K-1}$ of learning rates; coefficients of the problem; a number of indices $N$

\textbf{Output:} Approximation of $\theta^*$ minimizing $\mathbb{J}^N$, see~\eqref{eq:J-NN-initonly}
\begin{algorithmic}[1]
\FOR{$k = 0, 1, 2, \dots, K-1$}
    \STATE{Sample $\mathbf{S} = (x^i)_{i \in \llbracket N \rrbracket}$}
    \STATE{Solve the FFODE finite system~\eqref{eq:FFODE-finite-initonly} with current parameters $\theta^{(k)}$}
    \STATE{Compute the gradient $\nabla_{\theta} \mathbb{J}^N(\theta^{(k)},\mathbf{S})$ of $\mathbb{J}^N(\theta^{(k)},\mathbf{S})$ defined in~\eqref{eq:J-NN-Sample-initonly} }
    \STATE{Set $\theta^{(k+1)} =  \theta^{(k)} -\beta_k \nabla_{\theta} \mathbb{J}^N(\theta^{(k)},\mathbf{S})$ }
\ENDFOR
\RETURN $\theta^{(K)}$
\end{algorithmic}
\end{algorithm}

\subsection{Piecewise Constant Graphon}
Let $m^1, m^2,\dots, m^K$ be non-negative numbers such that $\sum_{k=1}^K m^k =1$.
We divide the player population into $K$ groups, $B^1,\dots, B^K$, placing all players with index $x\in [0,m^1)$ into group $B^1$, etc. We assume that players belonging to the same group are indistinguishable. For example, all players within a group must have the same recovery rate $\gamma$ and if $x,x'\in I$ are the indices of two players in the same group, then $w(x,y) = w(x',y)$ for all $y\in I$. In this situation, we only need to specify the graphon's values on each block of indices corresponding to a group, since the graphon is a constant on each block. Let us identify the group $B^i$ with its index block (or set). We can compactly represent the interaction weights between the block with a connection matrix $[w_{ij}]_{i,j=1}^K$, where $w_{ij}$ is the connection strength between players in block $B^i$ and players in block $B^j$. Then, for all players $x$ in block $B^i$
\begin{equation}
    Z_t^x =  \int_I w(x,y)\phi^y(t,\mathsf I)p^y(t,\mathsf I)dy = \sum_{k=1}^K w_{ik} \lambda^{(\mathsf{I},k)} p^k(t,\mathsf{I}) m^k,
\end{equation}
which is constant over $x\in B^i$. This is a feature of that we sometimes see when the piecewise constant graphon is used. Furthermore, we assume $\beta$, $\gamma$, $\kappa$ constant over each block but can differ between blocks. It opens up the possibility for us to solve the graphon game with classical numerical methods\footnote{These methods are referring to solving an ODE system similar to the one given in Section~\ref{sec:example1} for a finite (and possibly a small) number of $x$'s.} and a way of evaluating the \texttt{DeepGraphonGame} algorithm.

Turning to the remaining example set up, we specify the cost structure as a particular case of the general formulation of Section~\ref{sec:example1}. We assume that the regulator has set $\boldsymbol\lambda^{(\mathsf{S},k)}, \boldsymbol\lambda^{(\mathsf{I},k)}, \boldsymbol\lambda^{(\mathsf{R},k)}$, $k=1,\dots, K$, different for each block.

In this scenario, we first study the policy effects on the death ratio in the age groups. The policies compared are no lockdown; quarantine for infected; age specific lockdown; full lockdown. We can see in Figure~\ref{fig:ageblocks} that the death ratio decreases nearly 30\% if infected individuals are quarantined, compared to no lockdown. Furthermore, if an age specific lockdown is implemented, we see that even more lives are saved while not deteriorating the economy. Zooming in on the comparison of no lockdown and quarantine for infected (second row of Figure~\ref{fig:ageblocks}), we note that susceptible individuals are using a smaller contact factor when there is no quarantine in place. They optimize their risk and hence want to be more cautious, since the risk of getting infected is higher. 
\begin{table}[ht]
\small
\centering
\begin{tabular}{l|cccc}
Age   & \multicolumn{1}{l}{0-20} & \multicolumn{1}{l}{20-45} & \multicolumn{1}{l}{45-65} & \multicolumn{1}{l}{65+}  \\ 
\hline
0-20  & 1.0   & 0.9   & 0.8   & 0.7  \\
20-45 & 0.9   & 0.9   & 0.8   & 0.8  \\
45-65 & 0.8   & 0.8   & 0.9   & 0.8  \\
65+   & 0.7   & 0.8   & 0.8   & 0.8    
\end{tabular}\hskip1cm
\begin{tabular}{l|ccccc}
Age   & \multicolumn{1}{c}{$\beta$} & \multicolumn{1}{c}{$\gamma$} & \multicolumn{1}{c}{$\rho$} & \multicolumn{1}{c}{$p_0(\mathsf S)$}  & \multicolumn{1}{c}{$m$}\\ 
\hline
0-20  & 0.4   & 0.1   & 1.0   & 0.95  &  0.27\\
20-45 & 0.3   & 0.1   & 1.0   & 0.97  &  0.33\\
45-65 & 0.3   & 0.05  & 0.9   & 0.97  &  0.27\\
65+   & 0.3   & 0.05  & 0.75  & 0.97  &  0.13
\end{tabular}
\caption{\small Parameters of the experiment with different age groups}
\end{table}

\begin{figure}[H]
\centering
\begin{subfigure}{\textwidth}
\centering
    \includegraphics[width=0.6\linewidth]{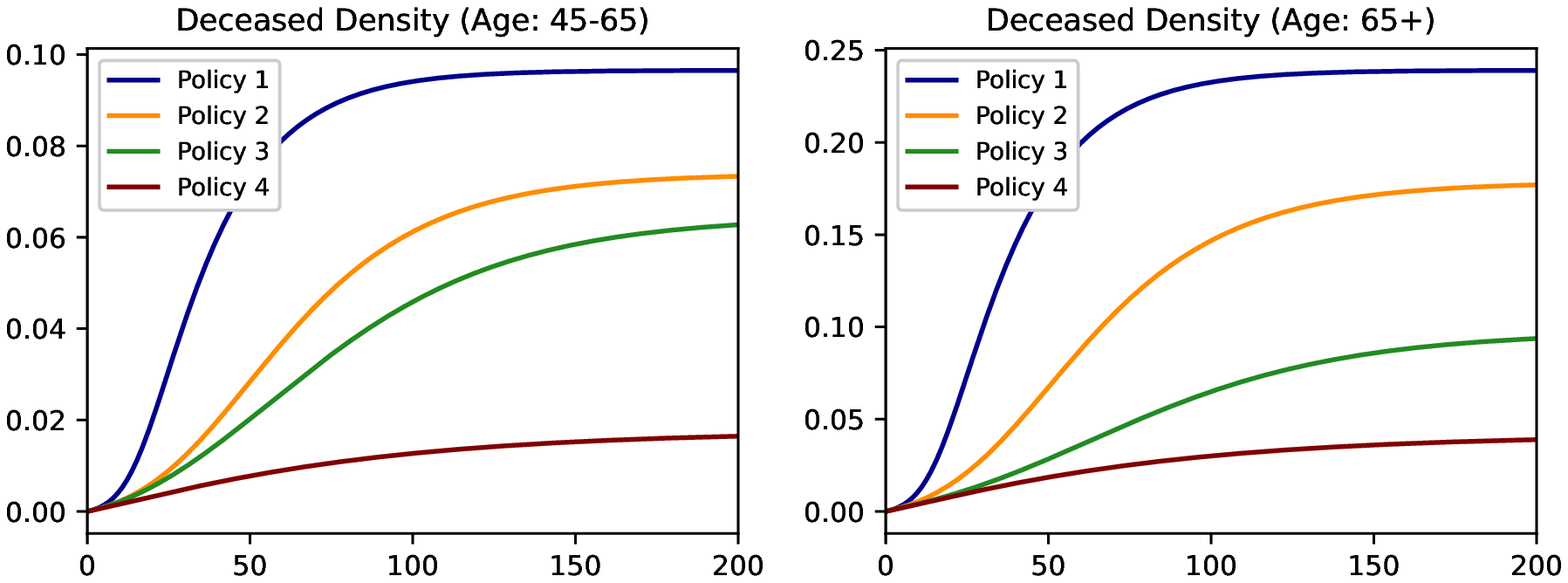}
\end{subfigure}
\begin{subfigure}{\textwidth}
\centering
    \includegraphics[width=\linewidth]{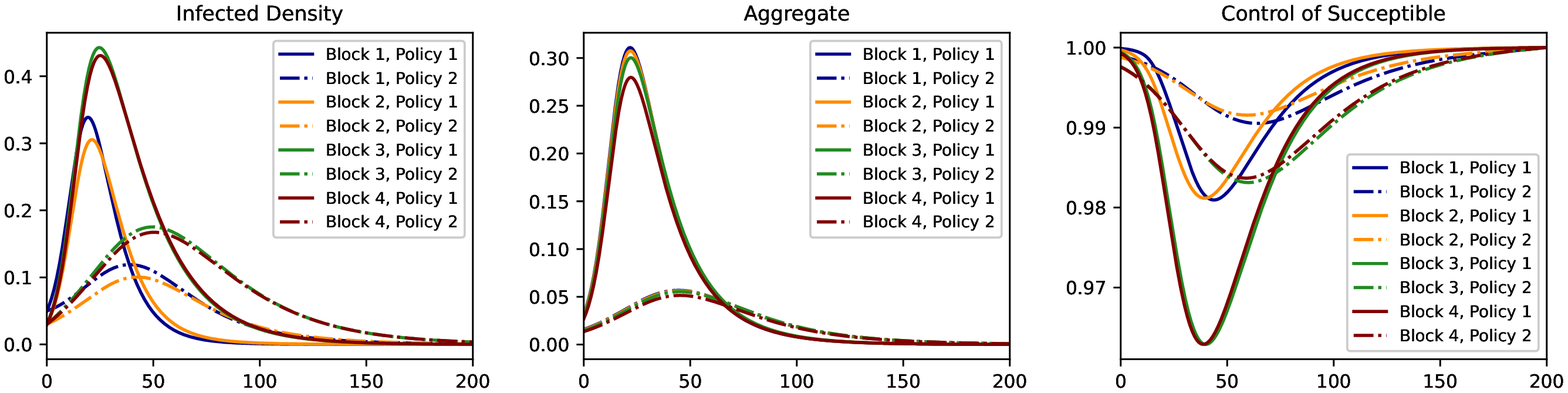}
\end{subfigure}
\caption{\small \textbf{Top:} Density of deceased people for age groups 45-65 (left) and 65+ (right) under different policies. \textbf{Bottom:} Comparison Plots for Policy 1 (no lockdown) and Policy 2 (qurantine for infected): Density of infected people (left), Aggregate $\boldsymbol Z$ (middle), Control of susceptible people (right) are plotted for each 4 blocks under both policies.}
\label{fig:ageblocks}
\end{figure} 

Secondly, in the same scenario, we model multiple cities with different attributes and study the effects of the travel restrictions. In this experiment, we compare a universal no-travel policy to the policy where traveling in or out of one of the cities is restricted (totaling four policies in the comparison). City 1 is a highly populated city with a more contagious virus variant, city 2 also has this variant; however, it is a small city. City 3 is a highly populated city but with a less contagious virus variant. For visual simplification, we assume that there are no deaths (\textit{i.e.} $\rho=0$). In Figure~\ref{fig:cityblocks}, we can see that the infected-density curve can be flattened the most if city 1 has travel restrictions. The reason is the existence of the more contagious variant and the large size of the city 1. We note that when this restriction is implemented the susceptible individuals feel relieved and increase their contact factor control.

\begin{table}[ht]
\small
\centering
\begin{tabular}{c|ccc}
Block    & \multicolumn{1}{c}{City 1} & \multicolumn{1}{c}{City 2} & \multicolumn{1}{c}{City 3} \\ 
\hline
City 1   & 0.3   & 0.3   & 0.3  \\
City 2   & 0.3   & 1.0   & 0.7  \\
City 3   & 0.3   & 0.7   & 1.0  \\
\end{tabular}\hskip10mm
\begin{tabular}{l|ccc}
Block   & \multicolumn{1}{c}{$\beta$} & 
\multicolumn{1}{c}{$p_0(\mathsf S)$} & 
\multicolumn{1}{c}{$m$} \\ 
\hline
City 1 & 0.4    & 0.95 &0.4\\
City 2 & 0.4    & 0.95 &0.2\\
City 3 & 0.3    & 0.95 &0.4
\end{tabular}
\caption{\small Connection matrix (left) and Parameters (right) used in the experiment with different cities. The connection matrix shown here is the case where there is a lockdown in City 1. When there is no lockdown the interaction weights are as follows: $w(\text{1, 1})= 1.0$, $w(\text{1, 2})= 0.9$ and $w(\text{1, 3})= 0.8$.}
\end{table}

\begin{figure}[H]
\centering
\begin{subfigure}{0.24\textwidth}
\centering
    \includegraphics[width=\linewidth]{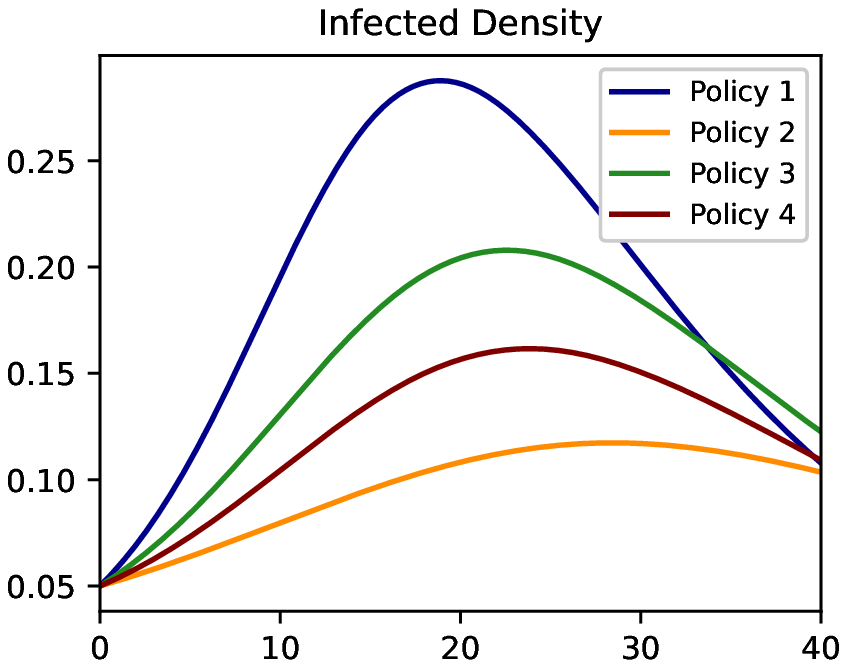}
\end{subfigure}
\begin{subfigure}{0.73\textwidth}
\centering
    \includegraphics[width=\linewidth]{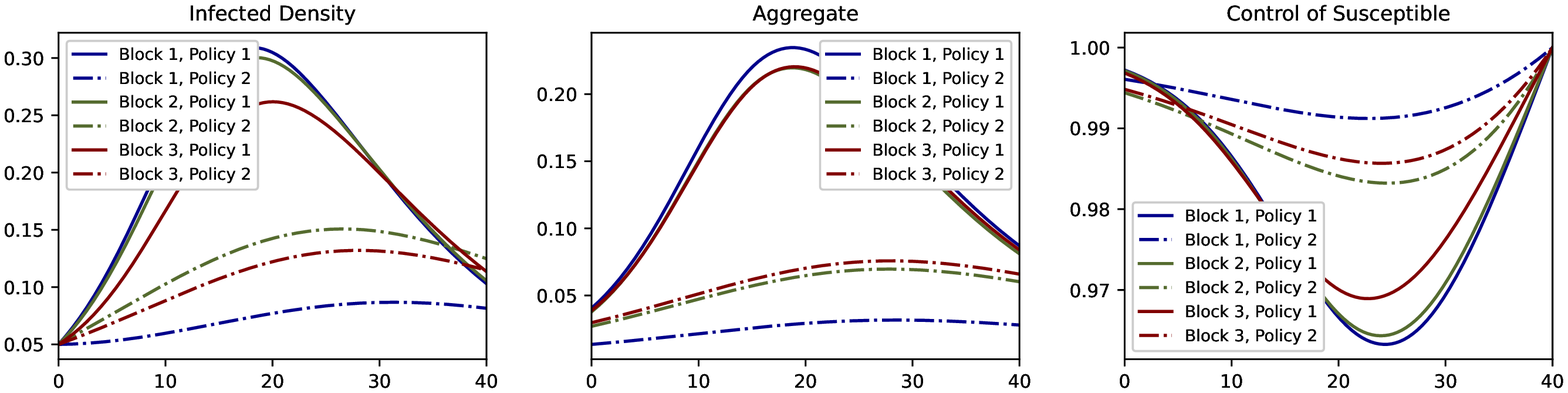}
\end{subfigure}
\caption{\small Density of infected people in the whole population (including all cities) under 4 different policies (left). Comparison Plots for Policy 1 (no lockdown) and Policy 2 (City 1 lockdown): Density of infected people in each city (\textit{i.e.} blocks) (middle left), Aggregate $\boldsymbol Z$ in each city (middle right), Control of susceptible people in each city (right).}
\label{fig:cityblocks}
\end{figure}

\subsubsection{Sanity Check for the Numerical Approach}

Here, we test the \texttt{DeepGraphonGame} algorithm by comparing its solution to the solution obtained by solving the ODE system for the cities-example when city 1 has travel restrictions. As can be seen in Figure \ref{fig:nnVSode}, the \texttt{DeepGraphonGame} algorithm approximates the exact result well. A plot of the function $x\mapsto u^x(0,\mathsf S)$, where $u^x$ is the numerically computed value function, can be seen on the right side of the bottom row in Figure~\ref{fig:nnVSode}. We can clearly see that agents in the same block have the same $u^x(0,\mathsf S)$ values. From this we infer that the \texttt{DeepGraphonGame} algorithm is preforming well when learning this piecewise constant function.

\begin{figure}[t!]
\centering
\begin{subfigure}{\textwidth}
\centering
    \includegraphics[width=\linewidth]{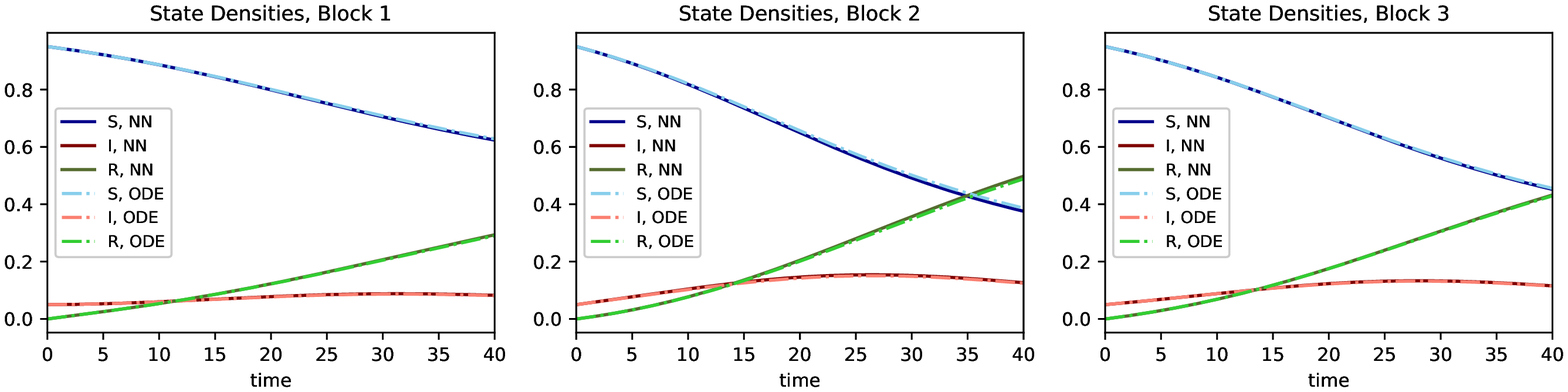}
\end{subfigure}

\begin{subfigure}{\textwidth}
\centering
    \includegraphics[width=\linewidth]{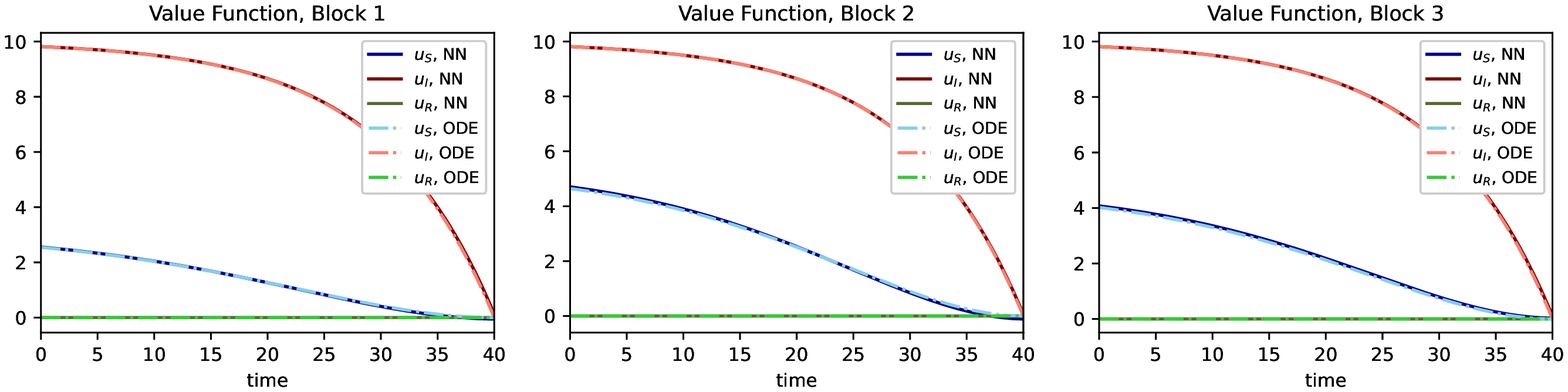}
\end{subfigure}

\begin{subfigure}{\textwidth}
\centering
    \includegraphics[width=\linewidth]{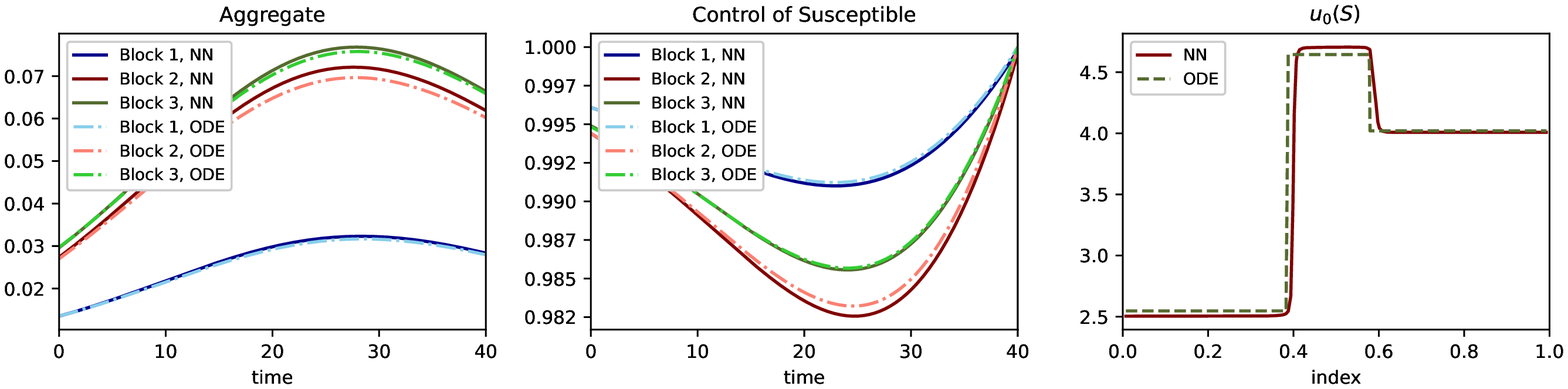}
\end{subfigure}
\caption{\small Comparison of the ODE and NN results when there is a lockdown for City 1: \textbf{Top:} State densities in Block 1 (left), Block 2 (middle) and Block 3 (right). \textbf{Middle:} Value functions given each state in Block 1 (left), Block 2 (middle) and Block 3 (right). \textbf{Bottom:} Aggregate $\boldsymbol Z$ (left), Control of susceptible people (middle) and value function at time 0 given state is susceptible as a function of index, $u^x(0, \mathsf S)$ (right). In the plots, Block $i$ refers to the City $i$.}
\label{fig:nnVSode}
\end{figure} 

\subsection{General Graphon}

To show scalability of the proposed numerical approach now we focus on the second example in \cite{aurell2020optimal} with the SEIRD model where the state $(\mathsf{E})$xposed is added. An individual is in state $\mathsf E$ when infected but is not yet infectious. Hence, the agents evolve from $\mathsf S$ to $\mathsf E$ and then $\mathsf I$, and the infection rate from $\mathsf S$ to $\mathsf E$ depends on the proportion of the infected agents. The diagram of the dynamics can be seen in Figure \ref{fig:SEIRD-diagram}. The cost structure is similar to the one used in Section~\ref{sec:example1}. After introducing the state $\mathsf{E}$, we set
\begin{equation*}
\begin{aligned}
\label{eq:SEIRD_cost}
    &f^x(t,e,z,\alpha) 
    =
    \frac{c_\lambda}{2}\left(\lambda^{(\mathsf S)}(t) - \alpha\right)^2 \mathbf{1}_{(e = \mathsf S)} + \frac{c_\lambda}{2}\left(\lambda^{(\mathsf E)}(t) - \alpha\right)^2 \mathbf{1}_{(e = \mathsf E)}\, + 
    \\
    &
     \left(\frac{1}{2}\left(\lambda^{(\mathsf I)}(t) - \alpha\right)^2 + c_I(x)\right)\mathbf 1_{(e = \mathsf I)} 
    + \frac{1}{2}\left(\lambda^{(\mathsf R)}(t)-\alpha\right)^2 \mathbf 1_{(e = \mathsf R)}
    + c_D(x) \mathbbm 1_{(e = \mathsf D)}.
\end{aligned}
\end{equation*}

\begin{figure}[t!]
\begin{center}

\tikzset{every picture/.style={line width=0.75pt}}

\begin{tikzpicture}[x=0.65pt,y=0.65pt,yscale=-1,xscale=1]
\draw  [fill={rgb, 255:red, 255; green, 165; blue, 0 }  ,fill opacity=1 ] (155,126.8) .. controls (155,123.6) and (157.6,121) .. (160.8,121) -- (179.2,121) .. controls (182.4,121) and (185,123.6) .. (185,126.8) -- (185,144.2) .. controls (185,147.4) and (182.4,150) .. (179.2,150) -- (160.8,150) .. controls (157.6,150) and (155,147.4) .. (155,144.2) -- cycle ;
\draw  [fill={rgb, 255:red, 255; green, 0; blue, 0 }  ,fill opacity=1 ] (240,126) .. controls (240,122.69) and (242.69,120) .. (246,120) -- (264,120) .. controls (267.31,120) and (270,122.69) .. (270,126) -- (270,144) .. controls (270,147.31) and (267.31,150) .. (264,150) -- (246,150) .. controls (242.69,150) and (240,147.31) .. (240,144) -- cycle ;
\draw  [fill={rgb, 255:red, 154; green, 205; blue, 50 }  ,fill opacity=1 ] (360,155.8) .. controls (360,152.6) and (362.6,150) .. (365.8,150) -- (384.2,150) .. controls (387.4,150) and (390,152.6) .. (390,155.8) -- (390,173.2) .. controls (390,176.4) and (387.4,179) .. (384.2,179) -- (365.8,179) .. controls (362.6,179) and (360,176.4) .. (360,173.2) -- cycle ;
\draw    (185,135) -- (238,135) ;
\draw [shift={(240,135)}, rotate = 180] [color={rgb, 255:red, 0; green, 0; blue, 0 }  ][line width=0.75]    (10.93,-3.29) .. controls (6.95,-1.4) and (3.31,-0.3) .. (0,0) .. controls (3.31,0.3) and (6.95,1.4) .. (10.93,3.29)   ;
\draw    (270,135) -- (358.25,85.97) ;
\draw [shift={(360,85)}, rotate = 510.95] [color={rgb, 255:red, 0; green, 0; blue, 0 }  ][line width=0.75]    (10.93,-3.29) .. controls (6.95,-1.4) and (3.31,-0.3) .. (0,0) .. controls (3.31,0.3) and (6.95,1.4) .. (10.93,3.29)   ;
\draw  [fill={rgb, 255:red, 230; green, 230; blue, 250 }  ,fill opacity=1 ] (360,75.8) .. controls (360,72.6) and (362.6,70) .. (365.8,70) -- (384.2,70) .. controls (387.4,70) and (390,72.6) .. (390,75.8) -- (390,93.2) .. controls (390,96.4) and (387.4,99) .. (384.2,99) -- (365.8,99) .. controls (362.6,99) and (360,96.4) .. (360,93.2) -- cycle ;
\draw    (270,135) -- (358.1,164.37) ;
\draw [shift={(360,165)}, rotate = 198.43] [color={rgb, 255:red, 0; green, 0; blue, 0 }  ][line width=0.75]    (10.93,-3.29) .. controls (6.95,-1.4) and (3.31,-0.3) .. (0,0) .. controls (3.31,0.3) and (6.95,1.4) .. (10.93,3.29)   ;
\draw  [fill={rgb, 255:red, 135; green, 206; blue, 250 }  ,fill opacity=1 ] (0,126.8) .. controls (0,123.6) and (2.6,121) .. (5.8,121) -- (24.2,121) .. controls (27.4,121) and (30,123.6) .. (30,126.8) -- (30,144.2) .. controls (30,147.4) and (27.4,150) .. (24.2,150) -- (5.8,150) .. controls (2.6,150) and (0,147.4) .. (0,144.2) -- cycle ;
\draw    (30,135) -- (153,135) ;
\draw [shift={(155,135)}, rotate = 180] [color={rgb, 255:red, 0; green, 0; blue, 0 }  ][line width=0.75]    (10.93,-3.29) .. controls (6.95,-1.4) and (3.31,-0.3) .. (0,0) .. controls (3.31,0.3) and (6.95,1.4) .. (10.93,3.29)   ;

\draw (163,129) node [anchor=north west][inner sep=0.75pt]   [align=left] {$\displaystyle E$};
\draw (250,129) node [anchor=north west][inner sep=0.75pt]   [align=left] {$\displaystyle I$};
\draw (368,158) node [anchor=north west][inner sep=0.75pt]   [align=left] {$\displaystyle R$};
\draw (75,115) node [anchor=north west][inner sep=0.75pt]  [font=\small] [align=left] {$\displaystyle \beta\alpha Z_t^x$};
\draw (303,157) node [anchor=north west][inner sep=0.75pt]  [font=\small] [align=left] {$\displaystyle \rho\gamma $};
\draw (368,78) node [anchor=north west][inner sep=0.75pt]   [align=left] {$\displaystyle D$};
\draw (8,129) node [anchor=north west][inner sep=0.75pt]   [align=left] {$\displaystyle S$};
\draw (280,80) node [anchor=north west][inner sep=0.75pt]  [font=\small] [align=left] {$\displaystyle (1-\rho)\gamma $};
\draw (210,120) node [anchor=north west][inner sep=0.75pt]  [font=\small]  [align=left] {$\displaystyle \epsilon $};

\end{tikzpicture}

\end{center}
\caption{Diagram of SEIRD model for individual $x$}
\label{fig:SEIRD-diagram}
\end{figure}
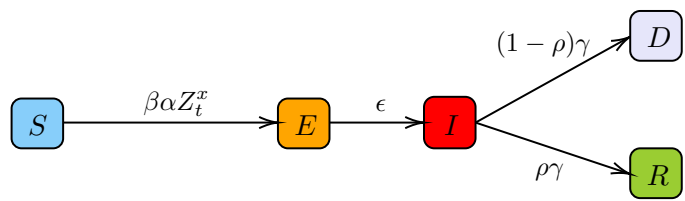

In this example, we focus on an application where the agents are not homogeneous over blocks. The interaction strength between individual $x$ and $y$ is now given by the power law graphon: $w(x,y) = (xy)^{-g}$ where $-\infty<g\leq0$ is a constant\footnote{We can realize that if $g=0$, the setting is equivalent to a mean field game.}. Intuitively, the power law graphon models interactions in a population where a small number of individuals are responsible for a large number of the interactions. For example, a population with superspreaders\footnote{A superspreader is an infected person who is able to transmit the disease to a disproportionately high number of people.} can be modeled with this graphon. The model with an underlying power law graphon interaction requires us to solve a continuum of coupled ODEs which is not computationally feasible. However, by using the \texttt{DeepGraphonGame} algorithm, the solution can be learned by using simulated particles, \textit{i.e.} agents. 

According to CDC, COVID-19 reinfection is very rare\footnote{\url{https://www.cdc.gov/coronavirus/2019-ncov/your-health/reinfection.html}}; therefore, we assume that there is no reinfection (\textit{i.e.} $\kappa =0$). Furthermore, the recovery duration is around 10 days since the symptom onset\footnote{\url{https://www.cdc.gov/coronavirus/2019-ncov/hcp/duration-isolation.html}}. For this reason, we assume that $\gamma=0.1\; \text{days}^{-1}$. According to the recent study conducted by Lauer \textit{et al.} \cite{incubation}, an exposed person begins to show symptoms after around 5 days. Based on this observation, we choose $\epsilon= 0.2\; \text{days}^{-1}$. Finally, the Basic Reproduction Number estimate $R_0=2$ used by CDC\footnote{\url{https://www.cdc.gov/coronavirus/2019-ncov/hcp/duration-isolation.html}} leads us to set $\beta = R_0 \times \gamma = 0.2$ in our simulations.

The experiment results for a sampled finite subset of the agent population are presented in Figure~\ref{fig:powerlaw}. In the figure, each line corresponds to one agent and the color of the plot gets darker as the index of the agent (\textit{i.e.} $x$) increases. Our first observation is that as the index of the agent increases, the aggregate also increases. In response to this high aggregate, the agent lower its contact rate, in order to protect itself. However, this protection is not enough to neutralize the effects of the high levels of the aggregate and the probability of the agent to be infected is still elevated.

\begin{figure}[t!]
\centering
\begin{subfigure}{\textwidth}
\centering
    \includegraphics[width=\linewidth]{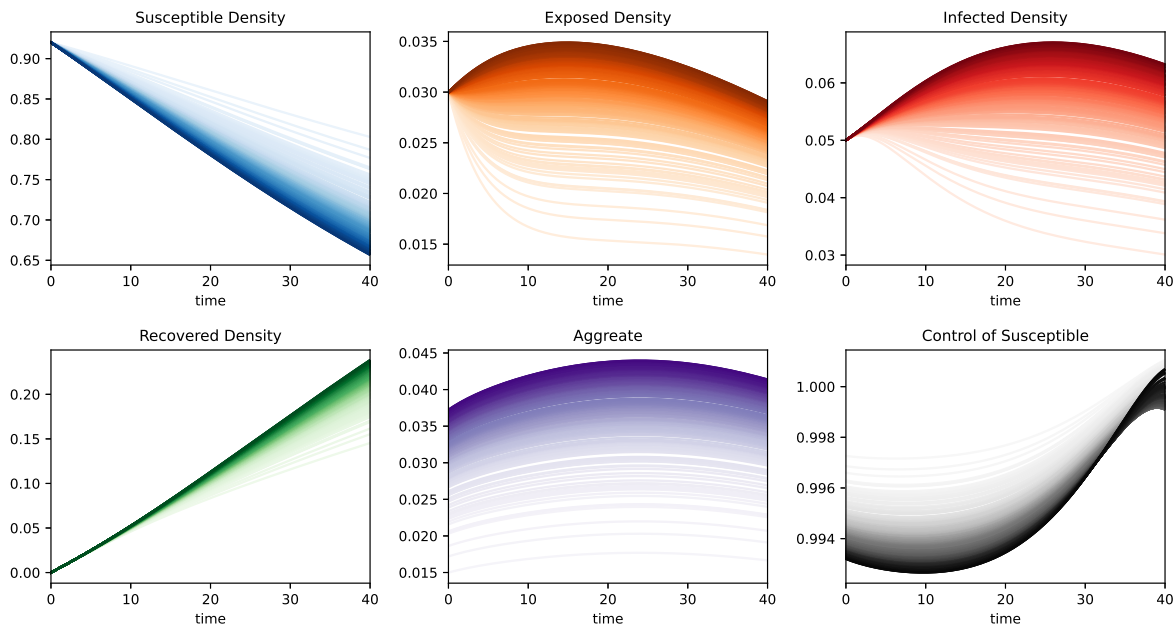}
\end{subfigure}
\caption{\small Results of agents from a sampled population. In each plot, colors are chosen from a continuous colormap to represent the index of the agents with the following convention: If the index $x$ of a player is higher, the color of the line is darker.\textbf{Top:} Probability of being susceptible (left), exposed (middle) and infected (right). \textbf{Bottom:} Probability of being recovered (left), Aggregate $\boldsymbol Z$ (middle), Control at the susceptible state (right).}
\label{fig:powerlaw}
\end{figure}

\section{The Probabilistic Approach to Epidemological Graphon Games}
\label{sec:fubini}

This section contains a closer study of the continuum of interacting jump processes that is the graphon game dynamics.
Going back to the informal discussion in the introductory Section~\ref{sec:intro-example}, it is not clear that there would be an adequate law of large numbers so that \eqref{eq:intro-example-1} converges to \eqref{eq:intro-example-2} since the random variables involved are dependent. A common approach used in the economic theory for this situation is to consider the continuum limit and average the continuum of random variables with respect to a non-atomic probability measure over $I$. Using the example from Section~\ref{sec:intro-example} again, a continuum limit of \eqref{eq:intro-example-1} is
\begin{equation*}
    \beta \alpha^{x_j}_t \int_I w(x_j, y)\alpha^y_t\mathbf{1}_{\mathsf{I}}(X^y_{t-})dy.
\end{equation*}
However, the integral in the expression above is ill-defined. There is an issue of constructing a continuum of independent random variables (here, that would be the driving Poisson noise) that are joint measurable in the sample and the index. If the construction is done in the usual way via the Kolmogorov construction then almost all random variables are essentially equal to an arbitrarily given function on the index space (\textit{i.e.}, as random variables they are constants). Hence, in any interesting case the function $y \mapsto X^y_{t-}$ will not be measurable with respect to the Lebesgue measure. One solution proposed by economists is to extend the usual probability space to a so-called a Fubini extension \cite{sun1998theory}, a probability space over $\Omega\times I$ where Fubini's theorem holds. On the Fubini extension a continuum of random variables can be constructed that are essentially pairwise independent (e.p.i.; see Theorem~\ref{thm:richfubini} for the definition) and jointly measurable in sample and index. Moreover, there is hope for an exact law of large numbers \cite{sun2006exact}, justifying our assumption about the determinism of the aggregate variable in the previous sections. We will construct a Fubini extension that carries a continuum of e.p.i. Poisson random measures. Then, each player path will be defined in the representation of the counting processes associated to the pure jump process as a stochastic integral with respect to a family of independent Poisson random measures with Lebesgue mean measure on $\mathbb{R} \times [0, T]$ as suggested in Skorokhod \cite{skorokhod1982studies} and Grigelionis \cite{grigelionis1971representation}. 

\subsection{Theoretical Background and Definitions}
\subsubsection{Poisson Random Measures}
\label{sec:Poisson}
Let us denote by $(\Gamma, \mathcal{G}, \mathbb{G})$ the measure space
$(\mathbb{R}\times[0,T], \mathcal{B}(\mathbb{R}\times[0,T]), \text{Leb}_{\mathbb{R}\times[0,T]})$. 
We first recall the definition of a Poisson random measure. A family $\bigl(N_\cdot(G)\bigr)_{G\in \mathcal{G}}$  of random variables defined on some probability space $(H, \mathcal{H}, \mathbb{H})$ 
is said to be a Poisson random measure with measure $\mathbb{G}$ if 
\begin{itemize}
    \item for all $G\in\mathcal{G}$ such that $\mathbb{G}(G)<\infty$, $N_\cdot(G)$ is a Poisson random variable with rate $\mathbb{G}(G)$;
    \item the random variables $N_\cdot(G_1),\dots, N_\cdot(G_n)$ are mutually independent whenever the sets $G_1,\dots, G_n \in \mathcal{G}$ have finite $\mathbb{G}$-measures and are disjoint;
    \item for all $\omega\in H$, $N_\omega(\cdot)$ is a measure on $(\Gamma,\mathcal{G})$.
\end{itemize}
We define 
$
\tilde{M} := \{\mu \ |\ \mu \text{ is a $\sigma$-finite non-negative measure on }(\Gamma,\mathcal{G})\},
$ 
and the subset of locally finite measures by 
$$
M := \{\mu\in\tilde{M}\ |\ \mu(\tilde G) < \infty \text{ for all bounded } \tilde G \in \mathcal{G}\}.
$$
For all bounded measurable $\tilde G$, define the mappings
$
I_{\tilde G} : \tilde{M} \ni \mu \mapsto \mu(\tilde G) \in \mathbb{R}.
$
Let $\tilde{\mathcal{M}}$ and $\mathcal{M}$ be the $\sigma$-algebra induced by the mappings $I_{\tilde G}$ on $\tilde{M}$ and $M$ respectively. For us, a random measure will be a measurable function from $(H, \mathcal{H}, \mathbb{H})$ into $(\tilde{M}, \tilde{\mathcal{M}})$ that almost surely takes values in $(M, \mathcal{M})$. We shall also use the fact that $M$ (equipped with the vague topology) is a Polish space \cite{dawson2012superprocesses}. We denote the law of $N$ on $(M, \mathcal{M})$ by $\mathcal{N}$. 

The Poisson random measure $N$ has an accompanying martingale. For all bounded $A\in \mathcal{B}(\mathbb{R})$, 
$\widehat{N}(A,t) := N(A\times [0,t])-  \text{Leb}(A)t$ is a square integrable zero-mean martingale.

\subsubsection{The Fubini Extension}
\label{sec:fub-ext}

In the model, state dynamics are given as $E$-valued jump processes. We will construct such a process from $2|E|-1=2n-1$ independent Poisson random measures. The possible jumps will be those so that the state process jumps between two integers in $E$, at most $n-1$ steps up or down. The initial state of player $x\in I$ is randomly selected from a pre-selected distribution $p^x_0 \in \mathcal{P}(E)$. 

In order to model idiosyncratic random shocks affecting the dynamics of the individual states, we use the framework of Fubini extensions. It allows us to capture a form of independence for a continuum of random variables, while preserving joint measurability.
\begin{definition}
\label{def:fub}
If $(\Omega, \mathcal{F}, \mathbb{P})$ and $(I, \mathcal{I}, \lambda)$ are probability spaces, a probability space $(\Omega\times I, \mathcal{W}, \mathbb{Q})$ extending the usual product space $(\Omega\times I, \mathcal{F}\otimes \mathcal{I}, \mathbb{P}\otimes \lambda)$ is said to be a Fubini extension if for any real-valued $\mathbb{Q}$-integrable function $f$ on $(\Omega\times I, \mathcal{W})$
\begin{enumerate}[label={(\roman*)}]
    \item\label{def:Fubini-i} the two functions $f_x : \omega \mapsto f(\omega,x)$ and $f_\omega : x \mapsto f(\omega,x)$ are integrable, respectively, on $(\Omega, \mathcal{F},\mathbb{P})$ for $\lambda$-a.e. $x\in I$, and on $(I,\mathcal{I},\lambda)$ for $\mathbb{P}$-a.e. $\omega\in \Omega$;
    \item\label{def:Fubini-ii} $\int_\Omega f_x(\omega) d\mathbb{P}$ and $\int_If_\omega(x)d\lambda(x)$ are integrable, respectively, on $(I,\mathcal{I},\lambda)$ and $(\Omega, \mathcal{F},\mathbb{P})$, with the Fubini property
    \begin{align*}
    \int_{\Omega\times I}f(\omega,x)d\mathbb{Q}(\omega,x) &= \int_I\left(\int_\Omega f_x(\omega)d\mathbb{P}(\omega)\right)d\lambda(x) 
    \\
    &= \int_\Omega\left(\int_If_\omega(x)d\lambda(x)\right)d\mathbb{P}(\omega).
    \end{align*}
\end{enumerate}
\end{definition}

The following theorem summarizes the results by Sun and collaborators, see for example \cite{sun2006exact} and \cite{sun2009individual}, which we use as a foundation for our model.

\begin{theorem}
\label{thm:richfubini}
There exists a probability space $(I,\mathcal{I},\lambda)$ extending $(I,\mathcal{B}_I,\lambda_I)$, a probability space $(\Omega, \mathcal{F},\mathbb{P})$, and a Fubini extension $(\Omega\times I, \mathcal{F}\boxtimes\mathcal{I}, \mathbb{P}\boxtimes\lambda)$ such that for any measurable mapping $\underline\phi$ from $(I,\mathcal{I}, \lambda)$ to $\mathcal{P}(E \times M^{2n-1})$ there is an $\mathcal{F}\boxtimes\mathcal{I}$-measurable process $\underline{f} : \Omega\times I \rightarrow E \times M^{2n-1}$ such that the random variables $f^x = \underline{f}(\cdot, x)$ are essentially pairwise independent (e.p.i.), \textit{i.e.}, for $\lambda$-a.e. $x\in I$, $f^x$ is independent of $f^y$ for $\lambda$-a.e. $y\in I$, and $\mathbb{P}\circ (f^x)^{-1} = \phi^x$ for all $x\in I$.
\end{theorem}

We set $\phi^x = p^x_0 \otimes (\otimes_{k=-n+1}^{n-1} \mathcal{N})$, where $\mathcal{N}$ is the probability law of the Poisson random measure introduced above, and $p^x_0$ is the initial distribution of player $x$. By Theorem~\ref{thm:richfubini} (which holds true since $E$ and $M$ are Polish spaces) there exists a collection of random variables $(\underline{\xi}, \underline{N}_{k}; k = -n+1 ,\dots, n-1)$ on a Fubini extension $(\Omega\times I, \mathcal{F}\boxtimes\mathcal{I}, \mathbb{P}\boxtimes \lambda)$, that are e.p.i. and $\phi^x$-distributed for all $x\in I$. With the model in Section~\ref{sec:model} in mind, we assume that the mapping $x \mapsto p^x_0$ is Lebesgue-measurable (this assumption is however not necessary for the analysis that follows).

We denote by $L^2_\boxtimes(\Omega\times I; \mathcal{D})$ the Bochner space of all (equivalence classes of) strongly $(\mathbb{P}\boxtimes\lambda,\mathcal{B}(\mathcal{D}))$-measurable functions $\underline{f} : \Omega\times I \rightarrow \mathcal{D}$ for which
\begin{equation*}
    \mathbb{E}^\boxtimes\left[\|\underline f\|^2_T\right] = \int_{\Omega\times I} \| \boldsymbol f^x(\omega)\|^2_T\mathbb{P}\boxtimes\lambda(d\omega, dx) < + \infty.
\end{equation*}
We define $L^2_\boxtimes(\Omega\times I; \mathcal{C})$ in the same way, with $\mathcal{C}$ replacing $\mathcal{D}$ above. By \textit{e.g.} \cite[Ch. 1.2.b]{hytonen2016analysis},  $L^2_\boxtimes(\Omega\times I; \mathcal{C})$ and $L^2_\boxtimes(\Omega\times I; \mathcal{D})$ are Banach spaces. 

For later reference, we define also the set $\mathcal{L}_E$ as the subset of $L^2_\boxtimes(\Omega\times I; \mathcal{D})$ of $\mathbb{P}\boxtimes\lambda$-a.e. $\mathcal{D}_E$-valued functions. One can show that $\mathcal{L}_E$ is a closed subset of $L^2_\boxtimes(\Omega\times I; \mathcal{D})$, hence $(\mathcal{L}_E, \|\cdot\|_{L^2_\boxtimes(\Omega\times I; \mathcal{D})})$ is a complete metric space.

\subsubsection{The Set of Admissible Strategies}
We can now give a rigorous definition of the set of admissible strategy profiles. Recall that $A$ is a compact subset of $\mathbb{R}$.
\begin{definition}
We define the set $\underline{\mathbb{A}}$ of admissible strategy profiles in feedback form as the set of $A$-valued, $\mathcal{I}\otimes \mathcal{B}([0,T]) \otimes \mathcal{B}(E)$-measurable functions $\underline{\alpha}$ on $I\times [0,T] \times E$ satisfying for every $(x
,e) \in I\times E$: $\underline{\alpha}(x,\cdot, e)$ is a continuous function on $[0,T]$.
\end{definition}
We will sometimes use the same notation $\underline{\mathbb{A}}$ for the set of admissible control processes associated to an admissible strategy profile in feedback form $\underline{\boldsymbol\alpha}$. At time $t$, for player $x\in I$, the value of such a control process is the action $\underline{\alpha}(x,t,X^{x,\boldsymbol{\alpha}}_{t-})$ where $X^{x,\boldsymbol{\alpha}}_{t-}$ is the state of player $x$ just before time $t$. These control processes are predictable with respect to the filtration generated by the player's own state, and decentralized  since they do not depend directly on the states of the other players. 

The continuity in time is a strong assumption and disallows the player to immediately react to abrupt changes in their environment without jumping to a new state. However, if a player transitions between two states at time $t$ their control can be discontinuous (as a multivariate function of time and state). In Section~\ref{sec:model}, a rationale was given for restriction our attention to such controls.

\subsection{The Finite State Graphon Game in the Fubini Extension}

We begin by describing an interacting system of a continuum of particles. First, we define the decoupled system where the aggregate variable vector has been ``frozen". Then, we define the aggregate with a fixed point argument. Finally, we prove that the aggregate is in fact deterministic. 

Consider the pure jump stochastic integral equation (here written formally)
\begin{equation}
\begin{aligned}
\label{eq:linear_dynamics_intro}
    &X_t^{\underline{\boldsymbol\alpha},\underline{\boldsymbol z},x} = \xi^x + \sum_{k=-n+1}^{n-1} k \int_{\mathbb{R}\times(0,t]} \mathbf{1}_{[0, 
    \kappa^x(X_{s-}^{\underline{\boldsymbol \alpha}, \underline{\boldsymbol z}, x},k, \alpha^x_s, z^x_{s-})]}(y)N_k^x(dy, ds),
\end{aligned}
\end{equation}
where $x\in I$, $t\in[0,T]$, $\underline{\boldsymbol\alpha}(\omega) = (\boldsymbol\alpha^x(\omega))_{x} := (\underline{\alpha}(x,t,X^{\underline{\boldsymbol\alpha},\underline{\boldsymbol z}, x}_{t-}(\omega))_{t,x}$ for some admissible strategy profile $\underline{\boldsymbol\alpha}$, $\underline{\boldsymbol z}\in L^2_\boxtimes(\Omega\times I; \mathcal{D})$, 
for any $k\in \mathbb{Z}$, $x\in I$, $s\in [0,T]$, $a\in A$, $z\in \mathbb{R}$, $i\in E$
\begin{equation}
    \kappa^x(i,k,a,z)
    :=
    \begin{cases}
    q^x_{i, i+k}(a,z), & i+k\in E,
    \\
    0, & i + k \not \in E,
    \end{cases}
\end{equation}
is the rate of jumps from state $i$ to state $i+k$ given the action $a\in A$ and the aggregate value $z$. The proposition below asserts that \eqref{eq:linear_dynamics_intro} has a unique solution in $\mathcal{L}_E$, the subset of $L^2_\boxtimes(\Omega\times I; \mathcal{D})$ defined in the end of section~\ref{sec:fub-ext}.

\begin{proposition}
\label{prop:X}
Assume that Condition~\ref{eq:cond_on_q} holds. Let $\underline{\boldsymbol z}\in L^2_\boxtimes(\Omega\times I; \mathcal{D})$ and $\underline{\alpha} \in \underline{\mathbb{A}}$ be fixed. Then there is a unique strong solution 
$\underline{\boldsymbol X}^{\underline{\boldsymbol\alpha},\underline{\boldsymbol z}} \in \mathcal{L}_E$ to \eqref{eq:linear_dynamics_intro}, \textit{i.e.}, a $\mathbb{P}\boxtimes\lambda$-a.e. $\mathcal{D}_E$-valued process satisfying \eqref{eq:linear_dynamics_intro} $\mathbb{P}\boxtimes\lambda$-a.s.
\end{proposition}
Note that the quantity
$$
M^x_k(t)=\int_{\mathbb{R}\times(0,t]} \mathbf{1}_{[0, 
    \kappa^x(X_{s-}^{\underline{\boldsymbol \alpha}, \underline{\boldsymbol z}, x},k, \alpha^x_s, z^x_{s-})]}(y)N_k^x(dy, ds),
$$
appearing in the right hand side of \eqref{eq:linear_dynamics_intro}, is a counting process with intensity $\kappa^x(X_{t-}^{\underline{\boldsymbol \alpha}, \underline{\boldsymbol z}, x},k, \alpha^x_t, z^x_{t-})$ at time $t\in[0,T]$ so by construction the solution to \eqref{eq:linear_dynamics_intro} (granted by Proposition~\ref{prop:X}) is almost surely an $E$-valued pure jump process with intensity matrix $Q^x(\alpha^x_t(\omega), z^x_{t-}(\omega))$ at time $t\in [0,T]$.

For a fixed admissible strategy profile $\underline{\boldsymbol\alpha}$, consider now the coupled system
\begin{equation}
\label{eq:solution-exists}
\begin{aligned}
    &X_t^{\underline{\boldsymbol\alpha},x} = \xi^x + \sum_{k=-n+1}^{n-1} k \int_{\mathbb{R}\times(0,t]} \mathbf{1}_{[0, 
    \kappa^x_s(X_{s-}^{\underline{\boldsymbol \alpha},x},k, \alpha^x_s, Z^{\underline{\boldsymbol\alpha},x}_{s-})]}(y)N_k^x(dy\otimes ds),
    \\
    &Z_t^{\underline{\boldsymbol\alpha},x} = \int_I w(x,y)K(\alpha^y_t, X^{\underline{\boldsymbol\alpha}, y}_{t-})\lambda(dy).
\end{aligned}
\end{equation}
The next theorem proves that \eqref{eq:solution-exists}
is well-posed with a unique solution in $L^2_\boxtimes$-sense. It further specifies the regularity of the solution: the aggregate variable $\underline{\boldsymbol Z}^{\underline{\boldsymbol\alpha}}$ must $\mathbb{P}\boxtimes\lambda$-a.s. be a deterministic and a continuous function of time.

\begin{theorem}
\label{thm:z-det-linear}
Let Condition~\ref{eq:cond_on_q} and \ref{cond:K_lip} hold, and let $\underline{\boldsymbol\alpha} \in \underline{\mathbb{A}}$.
\begin{itemize}
\item[(i)] There exists a unique solution $\underline{\boldsymbol X}^{\underline{\boldsymbol \alpha}}\in \mathcal{L}_E$ to \eqref{eq:solution-exists}. The corresponding aggregate $\underline{\boldsymbol Z}^{\underline{\boldsymbol \alpha}}$ is a random variable in $L^2_\boxtimes(\Omega\times I; \mathcal{C})$.
\\
\item[(ii)]
    The aggregate $\underline{\boldsymbol Z}^{\underline{\boldsymbol\alpha}}$ is $\mathbb{P}\boxtimes\lambda$-a.s. equal to a deterministic (\textit{i.e.}, constant in $\omega$) function in $L^2_\boxtimes(\Omega\times I; \mathcal{C})$. 
    \\
    \item[(iii)]
    There is a unique pair $\underline{\boldsymbol{\check X}}^{\underline{\boldsymbol\alpha}}$ and  $\underline{\boldsymbol {\check Z}}^{\boldsymbol{\underline\alpha}}$ of versions of $\underline{\boldsymbol X}^{\underline{\boldsymbol\alpha}}$ and $\underline{\boldsymbol Z}^{\underline{\boldsymbol\alpha}}$, respectively, solving \eqref{eq:solution-exists} for all $x\in I$ in the standard $L^2$-sense. Moreover $\underline{\boldsymbol {\check Z}}^{\boldsymbol{\underline\alpha}}$ is deterministic and continuous in time for all $x\in I$. 
\end{itemize}
\end{theorem}
Theorem~\ref{thm:z-det-linear} justifies working with a model defined for all $x\in I$ with a deterministic, continuous-in-time aggregate in Section~\ref{sec:model}. From here on, we will represent the $L^2_\boxtimes$-elements solving system \eqref{eq:solution-exists} with the version defined for all $x\in I$ and drop the check in the notation. 

\begin{remark}
If admissible strategy profiles did not have the prescribed continuity property we could not expect the aggregate to be a continuous function of time. One example of such a case is found in \cite{aurell2020optimal} where a regulator imposes a penalty that is discontinuous in time, resulting in equilibrium controls and aggregates discontinuous in time. We leave the analysis of the more general case to future work. 
\end{remark}

We now turn to the notion of player costs and equilibrium. Denote by $\mathbb{A}$ the set of $A$-valued and $\mathcal{B}([0,T])\otimes \mathcal{B}(E)$-measurable functions on $[0,T]\times E$, continuous $t\in [0,T]$ for every $e\in E$ that. If the player populations plays according to an admissible strategy profile $\underline{\alpha}\in \underline{\mathbb{A}}$ and player $x\in I$ decides to play strategy $\boldsymbol\sigma = (\sigma(t, X^{\underline{(\boldsymbol{\alpha}^{-x}, \boldsymbol \sigma)},x}_{t-}))_{t\in [0,T]}$ where $\sigma \in \mathbb{A}$  and 
\begin{equation*}
     (\boldsymbol{\alpha }^{-x},\boldsymbol\sigma)^y
     :=
     \begin{cases}
     \boldsymbol{ \alpha}^y, &\text{if }y\neq x
     \\
     \boldsymbol\sigma, &\text{if }y=x,
     \end{cases}
 \end{equation*}
then $\underline{(\boldsymbol{\alpha}^{-x},\boldsymbol\sigma)}$ is an admissible strategy profile and the player's expected cost for using $\boldsymbol \sigma$ is
\begin{equation*}
     \mathcal{J}^x(\boldsymbol\sigma;\boldsymbol{\underline{\alpha}}) :=
     \mathbb{E}
     \Big[
        \int_0^T f^x\big(t, X^{\underline{(\boldsymbol{\alpha}^{-x},\boldsymbol\sigma)},x}_t, \beta_t, Z^{\boldsymbol{\underline{\alpha}},x}_t\big)dt 
        + 
        h^x\big(X^{\underline{(\boldsymbol{\alpha}^{-x},\boldsymbol\sigma)},x}_T, Z^{\boldsymbol{\underline{\alpha}},x}_T\big)
     \Big].
 \end{equation*}
In fact, $\mathbb{A}$ is the set of strategies a player can deviate to without destroying the admissibility of the strategy profile. Therefore, we say that if $\underline{\boldsymbol{\hat \alpha}} = (\boldsymbol{\hat \alpha}^x)_{x\in I}\in \underline{\mathbb{A}}$ satisfies
\begin{equation*}
    \mathcal{J}^x(\boldsymbol{\hat\alpha}^x; \underline{\boldsymbol{\hat\alpha}}) 
    \leq 
    \mathcal{J}^x(\boldsymbol \sigma; \underline{\boldsymbol{\hat\alpha}}),\qquad \sigma \in \mathbb{A},\ x\in I,
\end{equation*}
the $\underline{\boldsymbol{\hat \alpha}}$ is a Nash equilibrium of the graphon game. The dependence of the cost on the whole strategy profile is unnecessarily complicated, as the following reasoning shows.
Notice that $\boldsymbol Z^{\underline{(\boldsymbol\alpha^{-x},\boldsymbol\sigma)}, x} = \boldsymbol Z^{\underline{\boldsymbol\alpha},x}$ since $\underline{(\boldsymbol\alpha^{-x},\boldsymbol\sigma)} = \underline{\boldsymbol\alpha}$ for $\lambda$-a.e. $x\in I$. The other players' actions appear in player $x$'s cost indirectly, through the aggregate $\boldsymbol Z^{\underline{\boldsymbol\alpha},x}$, which is unaffected if one specific player changes control (it is an integral with respect to a non-atomic measure).  
Thus, we write $\mathcal{J}^x(\boldsymbol\sigma;\underline{\boldsymbol\alpha})$ as $J^x(\boldsymbol\sigma;\boldsymbol Z^{\underline{\boldsymbol\alpha},x})$ a function taking an admissible strategy and an aggregate variable trajectory: 
$$
\mathbb{A} \times  \mathcal{C} \ni (\boldsymbol\sigma,\boldsymbol \zeta) \mapsto J^x(\boldsymbol\sigma; \boldsymbol \zeta) \in \mathbb{R}.
$$
 In light of this, an equivalent definition of the Nash equilibrium is that a strategy profile $\underline{\boldsymbol{\hat \alpha}} = (\boldsymbol{\hat \alpha}^x)_{x\in I}$ is a Nash equilibrium in the graphon game if it satisfies
\begin{equation*}
    J^x(\boldsymbol\alpha^x; \boldsymbol Z^{\underline{\boldsymbol\alpha},x}) \leq J^x(\boldsymbol \sigma; \boldsymbol Z^{\underline{\boldsymbol\alpha},x}),\qquad \boldsymbol\sigma \in \mathbb{A},\ x\in I,
\end{equation*}
further justifying the game setup in Section 2.\newline

\emph{\textbf{Acknowledgements.}}
{The authors would like to thank Boualem Djehiche and Yeneng Sun for helpful discussions.}

\appendix

\section{Proofs for Section~\ref{sec:model}}

\subsection{Theorem~\ref{theorem:feedback-control}}
\label{sec:thm1-proof}
First, assume player $x$ uses control $\boldsymbol \sigma$ and, using the HJB equation~\eqref{eq:gg-HJB} and the definition of the minimized Hamiltonian~\eqref{eq:-gg-Hopt}, note that
\begin{align*}
    &  d\left(u^x(t, X_t^{\boldsymbol \sigma,\boldsymbol{Z}^x,x})+\int_0^t f(s, X^{\boldsymbol \sigma,\boldsymbol{Z}^x,x}_s, Z_s^x, \sigma_s)ds \right) = M_t
    \\
    &
    + H^x(t,X_t^{\boldsymbol \sigma,\boldsymbol{Z}^x,x}, Z_t^x,  u^{x}(t,\cdot),\sigma_t) 
    -H^x(t,X_t^{\boldsymbol \sigma,\boldsymbol{Z}^x,x}, Z_t^x,  u^{x}(t,\cdot), \hat a_{X_t^{\boldsymbol \sigma,\boldsymbol{Z}^x,x}}^x(t, Z_t^x, u^x(t, \cdot))),
\end{align*} 
where $\mathbf{M}$ is a zero-mean martingale. 
Hence, we write 
\begin{align*}
    &
    \mathbb{E}\Big[u^x(t, X_t^{\boldsymbol \sigma,\boldsymbol{Z}^x,x})+\int_0^t f(s, X^{\boldsymbol \sigma,\boldsymbol{Z}^x,x}_s, Z_s^x, \sigma_s) ds \Big]
    \\
    &
    = \mathbb{E}\left[  u^x(0, X_0^{\boldsymbol \sigma,\boldsymbol{Z}^x,x}) \right]
    + \int_0^t \mathbb{E}\Big[ 
    H^x(s,X_s^{\boldsymbol \sigma,\boldsymbol{Z}^x,x}, Z_s^x,  u^{x}(s,\cdot),\sigma_s)
    \\
    &\qquad\qquad\qquad
    -H^x(s,X_s^{\boldsymbol \sigma,\boldsymbol{Z}^x,x}, Z_s^x,  u^{x}(s,\cdot), \hat a_{X_s^{\boldsymbol \sigma,\boldsymbol{Z}^x,x}}^x(s, Z_s^x, u^x(s, \cdot)))
    \Big] ds. 
\end{align*}
Note that the second expectation in the right hand side is non-negative since $\hat a_e^x(t, z, h)$ minimizes $A \ni \alpha \mapsto H^x(t,e,z,h, \alpha)$ by assumption. In fact, since it is the unique minimizer, this term is strictly positive unless $\sigma_t = \hat a_{X_s^{\boldsymbol \sigma,\boldsymbol{Z}^x,x}}^x(t, Z_t^x, u^x(t, e))$. 

By taking the expectation and by recalling the terminal condition we deduce that: 
\begin{align*}
    J^x(\boldsymbol\sigma; \boldsymbol Z^{x}) 
    &= \mathbb{E}\Big[u^x(T, X_T^{\boldsymbol \sigma,\boldsymbol{Z}^x,x})+\int_0^T f(s, X^{\boldsymbol \sigma,\boldsymbol{Z}^x,x}_s, Z_s^x, \sigma_s) ds \Big] 
    \\
    &\geq \mathbb{E}\big[u^x(0, X_0^{x})\big] = J^x(\hat{\boldsymbol \phi}^x; \boldsymbol Z^{x}),
\end{align*}
where, for the last equality, we used the interpretation of $u^x$ as player $x$'s value function.
Furthermore the inequality above is an equality if and only if $\sigma_t = \hat a_{X_s^{\boldsymbol \sigma,\boldsymbol{Z}^x,x}}^x(t, Z_t^x, u^x(t, e))$.

\subsection{Regularity of the Optimal Control}
\label{app:lemma1}
Here we show that $\hat a$ is continuous in $(t,z,h)$. 

\begin{lemma}
\label{lem:reg-hat-a}
    Assume Condition~\ref{cond:cond_differentiability}.\ref{cond:diff_q} and~\ref{cond:cond_differentiability}.\ref{cond:diff_f} hold. For every $x \in I$ and $e \in E$, $(t,z,h) \mapsto \hat a^x_e(t,z,h)$ defined by the Hamiltonian minimizer in~\eqref{eq:-gg-Hopt} is continuous. Moreover $(t,z,h) \mapsto \hat a^x_e(t,z,h)$ is locally Lipschitz continuous, \textit{i.e.}, for every positive constants $C_z$ and $C_h$, for every $e \in E,$ $a \in A,$  the function $(t,z,h) \mapsto \hat a^x_e(t,z,h)$ is Lipschitz continuous on $[0,T] \times [-C_z,C_z] \times [-C_h, C_h]^{|E|}$ (with a Lipschitz constant possibly depending on $T, C_z, C_h$). 
\end{lemma}

\begin{proof}
By the assumption on the dependence of $H^x$ on $a$, $\hat a^x_e(t,z,h) = \hat a^x(t,e,z,h)$ is the unique solution of the variational inequality (with unknown $a$): 
$$
    \forall b \in A, \qquad (b-a) \partial_a H^x(t,e, z, h, a) \ge 0.
$$
Let $C_z$ and $C_h$ be positive constants. Let $\theta = (t,e,z,h)$ and $\theta' = (t',e,z',h')$ with $(t,z,h)$, $(t',z',h') \in [0,T] \times [-C_z,C_z] \times [-C_h, C_h]^{|E|}$. Let $a = \hat a^x(\theta)$ and $a' = \hat a^x(\theta')$. We deduce from the above inequality that: 
$$
    (a'-a)[\partial_a H^x(\theta,a') - \partial_a H^x(\theta, a)]
    \le
    (a'-a)[\partial_a H^x(\theta,a') - \partial_a H^x(\theta', a')].
$$
By the assumption on the strict convexity of $f$, we have 
$$
    \lambda |a' - a|^2
    \le
    (a'-a)\big(\partial_a f^x(t,e,z,a') - \partial_a f^x(t,e,z,a)\big).
$$
So, combining the above inequalities and the property $|\partial_a Q^x(a,z) - \partial_a Q^x(a',z)| = 0$, we get
\begin{align*}
    \lambda |a' - a|^2
    &\le
    (a'-a)\big(\partial_a f^x(t,e,z,a') - \partial_a f^x(t,e,z,a)\big)
    \\
    &= (a' - a) \left[ \partial_a H^x(\theta,a') - \partial_a H^x(\theta,a)\right]
    \\
    &\le (a'-a)\left[\partial_a H^x(\theta,a') - \partial_a H^x(\theta', a')\right]
    \\
    &\le C |a'-a| \, \left(\sup_{e^{\prime} \in E}\left|\overrightarrow{\mathbf{1}_e} \partial_a Q^x(a',z)_{e^{\prime}} - \overrightarrow{\mathbf{1}_{e}} \partial_a Q^x(a',z')_{e^{\prime}}\right|  + \sup_{e^{\prime} \in E} |h_{e^{\prime}} - h'_{e^{\prime}}| \right)
    \\
    &\qquad + |a'-a| \, \left|\partial_a f^x(t,e,z,a') - \partial_a f^x(t',e,z',a')\right|,
\end{align*}
where $C$ can depend on $q_{\max}$, $C_z$, and $C_h$.
We conclude by using the continuity and local Lipschitz continuity propreties of $ \overrightarrow{\mathbf{1}_e} \partial_a Q^x$ and $\partial_a f^x$.
\end{proof}

\subsection{Theorem~\ref{theorem:existence-ode}}

\textit{Step 1: Definition of the solution space}\\
We start by letting, for every $C_1>0$, $\mathcal{K}_{C_1}$ be the closed ball of continuous functions $(u,p)$ from $[0,T]$ into $L^2(I\times E)\times L^2(I\times E)$ such that for all $(t,x)\in [0,T]\times I$, $p^x(t,e)\ge 0$ for $e\in E$ and $\sum_{e\in E}p^x(t,e)=1$, and such that $(u,p)$ is bounded by $C_1$ in uniform norm. In other words, $\mathcal{K}_{C_1}$ is the subset of $C([0,T];L^2(I\times E)\times L^2(I\times E))$ for which the second component is a probability on $E$ and for which the uniform norm is bounded by $C_1$, whose value will be fixed in Step~5 below.
\\~\\
\textit{Step 2: Definition of the aggregate mapping}\\
For each $(u,p)\in\mathcal{K}_{C_1}$ we define the map $\Phi^{(u,p)}$ which takes $(Z^x_t)_{0\le t\le T, x\in I}$ into
$$
\Phi^{(u,p)}\bigl((Z^x_t)_{0\le t\le T, x\in I}\bigr)=\Bigl(\int_I w(x,y)\sum_{e\in E}K\bigl(\hat a(t,e,Z^y_t,u^y(t,\cdot)),e\bigr) p^y(t,e)\, dy\Bigr)_{0\le t\le T, x\in I}.
$$
We now prove that, if we choose the space of aggregates properly, $\Phi^{(u,p)}$  has a unique fixed point, say $(\hat Z^x_t)_{0\le t\le T, x\in I}$, which depends continuously in $(u,p)\in\mathcal{K}_{C_1}$. 
Indeed, let $z,\widetilde z\in \mathcal{Z}$ where
$$
\mathcal{Z} := \{f \in C([0,T]; L^2(I; \mathbb{R})) : |f_t^x| \leq C_K,\ t\in [0,T],\ \text{ a.e. }x\in I\}.
$$
In light of Condition~\ref{cond:K_lip}, $\Phi^{(u,p)}(\mathcal{Z}) \subset \mathcal{Z}$. Moreover, $\mathcal{Z}$ is a closed subset of the Banach space $C([0,T]; L^2(I;\mathbb{R}))$, hence a complete metric space. Using Cauchy-Schwarz inequality and the Lipschitz continuity of $K$ and $\hat a$ (given by Condition~\ref{cond:K_lip} and Lemma~\ref{lem:reg-hat-a}), we get
\begin{align*}
        &\sup_{t\in[0,T]}\int_I |\Phi^{(u,p)}(z)^y - \Phi^{(u,p)}(\widetilde z)^y|^2 dy 
        \leq
       C_\Phi(C_K) \sup_{t\in [0,T]}\int_I|z^y_t - \widetilde z_t^y|^2dy,
\end{align*}
where we used the fact that $|I| = 1$ and $C_\Phi(\cdot) := \|w\|_{L^2(I\times I)}L_K L_{\hat a}(\cdot)$.
Recall that we assume $C_\Phi(C_K)< 1$. With Banach fixed point theorem we conclude that there exists a unique fixed point in 
$\mathcal{Z}$
to $\Phi^{(u,p)}$. We denote it $\hat Z^{(u,p)} = (\hat Z^{(u,p),x}_t)_{0\le t\le T, x\in I}$.
\\~\\
\textit{Step 3: Solving the Kolmogorov equation}\\
Given $\hat Z^{(u,p)}$ and $u$, we solve the Kolmogorov equation and we get the solution $\hat p$. Existence and uniqueness of the solution $\hat p$ is provided by the Cauchy-Lipschitz-Picard theorem; see, \textit{e.g.}, \cite[Theorem 7.3]{MR2759829} (viewing $q$ as a linear operator acting on the Banach space $L^2(I\times E)$).  Furthermore, given Condition~\ref{eq:cond_on_q}.\ref{eq:cond_on_q-bdd}, the time derivative of $\hat p$ is bounded. Therefore, we conclude that $\hat p$ is equicontinuous.
\\~\\
\textit{Step 4: Solving the HJB equation}\\
Given $\hat Z^{(u,p)}$ and $\hat p$, we solve the HJB equation and we get the solution $\hat u$. Here again, existence and uniqueness of the solution $\hat u$ is provided by the Cauchy-Lipschitz-Picard theorem
(viewing $\hat H$ as a Lipschitz operator acting on the Banach space $L^2(I\times E)$).  
Furthermore, there is  a uniform bound on the time derivative of $\hat u$ since the Hamiltonian $\hat H$ is bounded given Condition~\ref{cond:cond_differentiability}.\ref{cond:bound_f_g} and Condition~\ref{eq:cond_on_q}.\ref{eq:cond_on_q-bdd}. Hence $\hat u$ is equicontinuous.  
\\~\\
\textit{Step 5: Application of Schauder's theorem}\\ 
Let us call $\Psi$ the mapping constructed by the above steps, namely, $\Psi: \mathcal{K}_{C_1}\ni (u,p)\mapsto (\hat u, \hat p)$. By steps 3 and 4 above, if we choose $C_1$ large enough, $\Psi$ maps $\mathcal{K}_{C_1}$ onto $\mathcal{K}_{C_1}$, so it is well-defined. Furthermore, by the same steps and the Arzela-Ascoli theorem $\Psi(\mathcal{K}_{C_1})$ is compact. Finally we argue the continuity as follows:

We first show the continuity of $\hat Z^{(u,p)}$ in $u$ and $p$. Consider a sequence $(u^n,p^n)_n$ such that $(u^n,p^n) \in \mathcal{K}_{C_1}$ for every $n$, and  $\lim_{n\rightarrow\infty}(u^n,p^n)=(u,p)$. We denote $Z^n := Z^{(u^n,p^n)}$ and prove below that 
$\lim_{n\rightarrow \infty}\hat Z^{n} = \hat Z$. 
By Lipschitz continuity of $K$ and $\hat a$:
\begin{equation*}
    \begin{aligned}
        &\sup_{t\in [0,T]} \int_I \big|\hat Z_t^{n,x}-\hat Z_t^x\big|^2dx  \leq
        \\
        &\sup_{t \in [0,T]} \int_I \left(C_\Phi(C_K) \big|\hat Z_t^{n,x}-\hat Z_t^x\big|^2 + C_p \big|p^{n,x}(t,\cdot)-p^x(t,\cdot)\big|^2 +C_u \big|u^{n,x}(t,\cdot)-u^x(t,\cdot)\big|^2 \right)dx,
    \end{aligned}
\end{equation*}
for some constants $C_p, C_u \ge 0 $, and where $C_\Phi(C_K) =\|w\|_{L^2(I\times I)}L_K L_{\hat a}(C_K)$. Hence
\begin{equation*}
    \begin{aligned}
    &\sup_{t\in [0,T]} \int_I \big|\hat Z_t^{n,x}-\hat Z_t^x\big|^2dx 
    \\
        &
        \leq\sup_{t \in [0,T]} \int_I \frac{1}{1-C_\Phi(C_K)}\Big(C_p \big|p^{n,x}(t,\cdot)-p^x(t,\cdot)\big|^2 +C_u \big|u^{n,x}(t,\cdot)-u^x(t,\cdot)\big|^2\Big)dx,
    \end{aligned}
\end{equation*}
which tends to $0$ as $n\rightarrow\infty$.

Next, we study the continuity of $\hat p$. We have, for $s \in [0,T]$:
\begin{align*}
    &\int_I \big|\hat p^{n,x}(s,\cdot)-\hat p^x(s,\cdot)\big|^2 dx
    \\
    &\le \int_I \int_0^s \sum_{e, e^{\prime} \in E} \Big( \Big|  q^x_{e^{\prime},e}(\hat a(t,e^{\prime},\hat Z^{n,x}_t,u^{n,x}(t,\cdot)), \hat Z_t^{n,x}) 
    -  q^x_{e^{\prime},e}(\hat a(t,e^{\prime},\hat Z^{x}_t,u^{x}(t,\cdot)), \hat Z_t^{x}) \Big|^2 |\hat p^{n,x}(t, e^{\prime})|^2
    \\
    &\qquad\qquad\qquad\qquad +  |q^x_{e^{\prime},e}(\hat a(t,e^{\prime},\hat Z^{x}_t,u^{x}(t,\cdot)), \hat Z_t^{x})|^2 \Big|\hat p^{n,x}(t, e^{\prime})  - \hat p^x(t, e^{\prime}) \Big|^2 \Big) dt dx
    \\
    &\le C  \Big( \int_0^s \sum_{e, e^{\prime} \in E} \int_I \Big|  q^x_{e^{\prime},e}(\hat a(t,e^{\prime},\hat Z^{n,x}_t,u^{n,x}(t,\cdot)), \hat Z_t^{n,x}) -  q^x_{e^{\prime},e}(\hat a(t,e^{\prime},\hat Z^{x}_t,u^{x}(t,\cdot)), \hat Z_t^{x}) \Big|^2 dx dt
    \\
    &\qquad\qquad\qquad\qquad + \int_0^s \int_I  \Big|\hat p^{n,x}(t,\cdot)  - \hat p^x(t,\cdot) \Big|^2 dx dt \Big),
\end{align*}
where we used Condition~\ref{eq:cond_on_q}.\ref{eq:cond_on_q-bdd} and a uniform  (\textit{i.e.}, independent of $n$) bound on $\hat p^{n,x}(t,\cdot)$ (since $q$ is bounded independently of $n$).
By Gr\"onwall's inequality, we obtain
\begin{align*}
    &\int_I \big|\hat p^{n,x}(s,\cdot)-\hat p^x(s,\cdot)\big|^2 dx
    \\
     &\le C  \int_0^s \sum_{e, e^{\prime} \in E} \int_I \Big|  q^x_{e^{\prime},e}(\hat a(t,e^{\prime},\hat Z^{n,x}_t,u^{n,x}(t,\cdot)), \hat Z_t^{n,x}) -  q^x_{e^{\prime},e}(\hat a(t,e^{\prime},\hat Z^{x}_t,u^{x}(t,\cdot)), \hat Z_t^{x}) \Big|^2 dx dt
\end{align*}
which tends to $0$ as $n\rightarrow\infty$, and we concluded by using the continuity of $q$ (given by Condition~\ref{eq:cond_on_q}.\ref{eq:cond_on_q-lip}).

Finally, we study the continuity of $\hat u$. For $s \in[0,T]$, we have:

\begin{align*}
        &\int_I \big|\hat u^{n,x}(s,\cdot)-\hat u^{x}(s,\cdot)\big|^2dx
        \\
        &\leq \int_I \int_s^T \big|-H^x(t,e,\hat Z_t^{n,x}, \Delta_e \hat u^{n,x}(t,\cdot))+H^x(t,e,\hat Z_t^{x}, \Delta_e \hat u^{x}(t,\cdot))\big|^2dtdx
        \\
        &\leq \int_s^T \int_I \sum_{e^{\prime}\in E} \Big|q_{e,e^{\prime}}\big(\hat a(t,e,\hat Z^{n,x}_t,\hat u^{n,x}(t,\cdot)), \hat Z_t^{n,x}\big)\Big|^2|\hat u^x(t,\cdot) - \hat u^{n,x}(t,\cdot)|^2dxdt 
        \\
        &\qquad + \int_I\int_s^T \Big|q_{e,e^{\prime}}\big(\hat a(t,e,\hat Z^{x}_t,\hat u^{x}(t,\cdot)), \hat Z_t^{x}\big)-q_{e,e^{\prime}}\big(\hat a(t,e,\hat Z^{n,x}_t,\hat u^{n,x}(t,\cdot)), \hat Z_t^{n,x}\big)\Big|^2|\hat u^x(t,\cdot)|^2dtdx
        \\
        & \qquad+  \int_I \int_s^T\Big| - f^x\big(t,e,\hat Z^{n,x}_t,\hat a(t,e,\hat Z^{n,x}_t,\hat u^{n,x}(t,\cdot))\big)+ f^x\big(t,e,\hat Z^{x}_t,\hat a(t,e,\hat Z^{x}_t,\hat u^{x}(t,\cdot))\big)\Big|^2dtdx
        \\
        &\leq \int_s^T \int_I \sum_{e^{\prime}\in E} \Big|q_{e,e^{\prime}}\big(\hat a(t,e,\hat Z^{n,x}_t,\hat u^{n,x}(t,\cdot)), \hat Z_t^{n,x}\big)\Big|^2|\hat u^x(t,\cdot) - \hat u^{n,x}(t,\cdot)|^2dxdt 
        \\
        &\qquad+ C\int_I\int_s^T \Big|q_{e,e^{\prime}}\big(\hat a(t,e,\hat Z^{x}_t,\hat u^{x}(t,\cdot)), \hat Z_t^{x}\big)-q_{e,e^{\prime}}\big(\hat a(t,e,\hat Z^{n,x}_t,\hat u_t^{n,x}(\cdot)), \hat Z_t^{n,x}\big)\Big|^2dtdx
        \\
        &\qquad +  \int_I \int_s^T\Big| - f^x\big(t,e,\hat Z^{n,x}_t,\hat a(t,e,\hat Z^{n,x}_t,\hat u^{n,x}(t,\cdot))\big)+ f^x\big(t,e,\hat Z^{x}_t,\hat a(t,e,\hat Z^{x}_t,\hat u^{x}(t,\cdot))\big)\Big|^2dtdx.
    \end{align*}
Using the boundedness of $\hat u^x$ and  $\hat u^{n,x}$ (given by Condition~\ref{cond:cond_differentiability}.\ref{cond:bound_f_g}) and the boundedness of $q$ (given by Condition~\ref{eq:cond_on_q}.\ref{eq:cond_on_q-bdd}), we deduce with Gr\"onwall's inequality:
\begin{equation*}
    \begin{aligned}
        &\int_I \big|\hat u^{n,x}(s,\cdot)-\hat u^{x}(s,\cdot)\big|^2dx\\
        &\leq C\Big(\int_I\int_s^T \Big|q_{e,e^{\prime}}\big(\hat a(t,e,\hat Z^{x}_t,\hat u^{x}(t,\cdot)), \hat Z_t^{x}\big)-q_{e,e^{\prime}}\big(\hat a(t,e,\hat Z^{n,x}_t,\hat u^{n,x}(t,\cdot)), \hat Z_t^{n,x}\big)\Big|^2dtdx\\
        &\qquad +  \int_I \int_s^T\Big| - f^x\big(t,e,\hat Z^{n,x}_t,\hat a(t,e,\hat Z^{n,x}_t,\hat u^{n,x}(t,\cdot))\big)+ f^x\big(t,e,\hat Z^{x}_t,\hat a(t,e,\hat Z^{x}_t,\hat u^{x}(t,\cdot))\big)\Big|^2dtdx\Big),
    \end{aligned}
\end{equation*}
which tends to zero as $n\rightarrow\infty$ by the continuity of $f$ and $q$ (given by Condition~\ref{cond:cond_differentiability}.\ref{cond:bound_f_g} and Condition~\ref{eq:cond_on_q}.\ref{eq:cond_on_q-lip} respectively) and the fact that $\lim_{n\rightarrow\infty}\hat Z^{n} = \hat Z$.

We conclude the proof by applying Schauder's theorem, see \textit{e.g.}, \cite[Excercise 6.26]{MR2759829}, which yields the existence of a fixed point $(\hat u, \hat p) \in \mathcal{K}_{C_1}$ to $\Psi$.

\section{Proofs for Section~\ref{sec:fubini}}

\subsection{Proposition~\ref{prop:X}}

The proof is inspired by \cite[Thm. 3.2]{choutri2019mean} where the authors study problems of optimal control of McKean-Vlasov-type pure jump processes and propose the use of the the identities found between \eqref{eq:diff-in-state-prop1} and \eqref{eq:diff-in-quad-var-prop1} below.

Consider the mapping $\Psi : \mathcal{L}_E \rightarrow \mathcal{L}_E$ defined by
\begin{equation*}
    \Psi(\underline{\boldsymbol\chi}) := 
    \underline\xi + \sum_{k=-n+1}^{n-1}k \int_{(0,\cdot]} \mathbf{1}_{[0,\kappa_s(\underline\chi_{s-}, k, \underline\alpha_s, \underline z_{s-})]}(y)\underline N_k(dy,ds),\qquad
    \underline{\boldsymbol\chi}\in \mathcal{L}_E.
\end{equation*}
We note that $\Psi$ is well-defined. Indeed, for any $\underline{\boldsymbol \chi}\in \mathcal{L}_E$, $\Psi(\underline{\boldsymbol \chi})$ is a linear combination of the initial condition $\underline\xi$ and the Poisson random measures evaluated at measurable sets, hence $\mathcal{F}\boxtimes \mathcal{I}$-measurable. By construction $\Psi(\underline{\boldsymbol \chi})$ is $\mathbb{P}\boxtimes\lambda$-a.e. $\mathcal{D}_E$-valued which implies the integrability.

To conclude the proof, we show that $\Psi$ has the contraction property, \textit{i.e.}, that there exists a constant $C<1$ such that
\begin{equation*}
    \mathbb{E}^\boxtimes\left[\|\Psi(\underline{\boldsymbol \chi}) - \Psi(\underline{\boldsymbol {\widetilde\chi} })\|^2_T\right] \leq 
    C\mathbb{E}^\boxtimes\left[\|\underline{\boldsymbol \chi} - \underline{\boldsymbol {\widetilde\chi} }\|^2_T\right], \quad \underline{\boldsymbol \chi}, \underline{\boldsymbol {\widetilde \chi}} \in \mathcal{L}_E.
\end{equation*}

By independence of the Poisson measures $(N^x_k)_{k=1}^n$ for any fixed $x\in I$, the compensated martingales $(\widehat{N}^x_k)_{k=1}^n$ (cf. Section~\ref{sec:Poisson}) are orthogonal. Hence, for any $\underline{\boldsymbol \chi} \in \mathcal{L}_E$ and $t\in [0,T]$,
\begin{align*}
      &\xi^x(\omega) + \sum_{k=-n+1}^{n-1}k \int_{(0,t]} \mathbf{1}_{[0,\kappa_s(\chi^x_{s-}(\omega), k, \alpha^x_s(\omega), z^x_{s-}(\omega))]}(y)N^x_k(\omega,dy,ds) 
      \\
      &= \xi^x(\omega) + \int_0^t\sum_{k=-n+1}^{n-1} k \kappa^x_s(\chi^x_s(\omega), k, \alpha^x_s(\omega), z^x_s(\omega))ds + M^{\underline{\boldsymbol\alpha},\underline{\boldsymbol z},x}_t(\omega),\quad \mathbb{P}\boxtimes\lambda\text{-a.e.}
\end{align*}
where $\underline{\boldsymbol M}^{\underline{\boldsymbol \alpha},\underline{\boldsymbol z}}$ is a zero-mean stochastic integral, and we have that
\begin{equation}
\label{eq:diff-in-state-prop1}
    \begin{aligned}
        &|\chi^x_t(\omega) - \widetilde\chi^x_t(\omega)|^2
        \\
        &\leq 
        C\left(\int_0^t \sum_{k=-n+1}^{n-1} |k||\kappa^x_s(\chi^x_s(\omega), k, \alpha^x_s(\omega), z^x_s(\omega)) - \kappa^x_s(\widetilde\chi^x_s(\omega), k, \alpha^x_s(\omega), z^x_s(\omega))|ds \right)^2
        \\
        &\qquad + C|M^{\underline{\boldsymbol\alpha},\underline{\boldsymbol z},x}_t(\omega) - \widetilde M^{\underline{\boldsymbol\alpha},\underline{\boldsymbol z},x}_t(\omega)|^2,\quad \mathbb{P}\boxtimes\lambda\text{-a.e.}
    \end{aligned}
\end{equation}
where the process $\underline{\boldsymbol{\widetilde M}}^{\underline{\boldsymbol\alpha},\underline{\boldsymbol z}}$ is the equivalent of $\underline{\boldsymbol{ M}}^{\underline{\boldsymbol\alpha},\underline{\boldsymbol z}}$ but for $\boldsymbol{\widetilde \chi}$. The predictable quadratic variation of the process $\underline{\boldsymbol M}^{\underline{\boldsymbol\alpha},\underline{\boldsymbol z}} - \underline{\boldsymbol {\widetilde M}}^{\underline{\boldsymbol\alpha},\underline{\boldsymbol{ z}}}$ is
\begin{align*}
&\langle    \boldsymbol M^{\underline{\boldsymbol\alpha},\underline{\boldsymbol z},x} - \boldsymbol {\widetilde M}^{\underline{\boldsymbol\alpha},\underline{\boldsymbol{z}},x} \rangle_t = 
\\
&
\sum_{k=-n+1}^{n-1}
k^2
\int_{\mathbb{R}\times(0,t]}\left(
\mathbf{1}_{[0, \kappa^x_s(\chi^x_s,k, \alpha^x_s, z^x_s)]}(y) - \mathbf{1}_{[0, \kappa^x_s(\widetilde \chi ^x_s ,k, \alpha^x_s, z^x_s)]}(y)\right)^2 dy\otimes ds,\quad \mathbb{P}\boxtimes\lambda\text{-a.e.}
\end{align*}
where the $\omega$-dependence has been suppressed in the notation.
Expanding the square and integrating, using the identities
\begin{align*}
    &\left(\mathbf{1}_{[0,a]} - \mathbf{1}_{[0,b]}\right)^2 
    = \mathbf{1}_{[0,a]} - 2\mathbf{1}_{[0,\min(a,b)]} + \mathbf{1}_{[0,b]},
\end{align*}
and $a - 2\min(a,b) + b \leq |a - b|$ for $a,b\geq 0$, we obtain
\begin{equation*}
\label{eq:cite-s-and-b}
    \langle \boldsymbol M^{\underline{\boldsymbol\alpha},\underline{\boldsymbol z},x} - \boldsymbol {\widetilde M}^{\underline{\boldsymbol \alpha},\underline{\boldsymbol{ z}},x} \rangle_t
    \leq
    \int_0^t\sum_{k=-n+1}^{n-1}k^2|\kappa^x_s(\chi^x_s,k, \alpha^x_s, z^x_s) - \kappa^x_s(\widetilde \chi^x_s, k, \alpha^x_s, z^x_s)|ds,\ \mathbb{P}\boxtimes\lambda\text{-a.e.}
\end{equation*}
From the Lipschitz continuity imposed by Condition~\ref{eq:cond_on_q}.\ref{eq:cond_on_q-lip}, we get
\begin{equation}
\label{eq:diff-in-quad-var-prop1}
\begin{aligned}
    \langle    \boldsymbol M^{\underline{\boldsymbol\alpha},\underline{\boldsymbol z},x}(\omega) - \boldsymbol {\widetilde M}^{\underline{\boldsymbol \alpha},\underline{\boldsymbol{ z}},x}(\omega) \rangle_t
    &\leq
    C\int_0^t \mathbf{1}_{\{\chi^x_s(\omega) \neq \widetilde\chi^x_s(\omega)\}}ds
    \\
    &\leq 
    C\int_0^t
    \|\boldsymbol \chi^x(\omega) - \boldsymbol{\widetilde \chi}^x(\omega)\|^2_s ds,\ \mathbb{P}\boxtimes\lambda\text{-a.e.}
    \end{aligned}
\end{equation}
After taking expectation of \eqref{eq:diff-in-state-prop1}, we get using Doob's inequality and \eqref{eq:diff-in-quad-var-prop1} that
\begin{equation*}
    \mathbb{E}^\boxtimes\left[\|\Psi(\underline{\boldsymbol \chi}) - \Psi(\underline{\boldsymbol{\widetilde \chi}})\|^2_T \right] 
    \leq 
    C \int_0^T
    \mathbb{E}^\boxtimes\left[\|\underline{\boldsymbol \chi} - \underline{\boldsymbol{\widetilde \chi}}\|^2_s \right] ds.
\end{equation*}
Iterating the inequality, we get for any $N\in \mathbb{N}$ that
\begin{equation*}
    \mathbb{E}^\boxtimes\left[\|(\Psi^N)(\underline{\boldsymbol \chi})) - \Psi(\underline{\boldsymbol{\widetilde \chi}})^{(N)}\|^2_T \right] \leq \frac{(CT)^N}{N!} 
    \mathbb{E}^\boxtimes\left[\|\underline{\boldsymbol \chi} - \underline{\boldsymbol{\widetilde \chi}}\|^2_T \right],
\end{equation*}
where $(\Psi^N)$ denotes the $N$-fold composition of $\Psi$. Thus, for some $N$ large enough, $(\Phi^N)$ is a contraction and if follows from the Banach fixed-point theorem for iterated mappings (see e.g. \cite{bryant1968remark}) that $\Psi$ has a unique fixed point in the set $\mathcal{L}_E$. The fixed point is the unique (up to $\mathbb{P}\boxtimes\lambda$-modification) strong solution that was sought.

\subsection{Theorem~\ref{thm:z-det-linear}}

The first step of the proof is to show that the aggregate variable is well-defined as a fixed point to the mapping
\begin{equation}
\label{eq:U}
\begin{aligned}
    &U^{\underline{\boldsymbol\alpha}} : L^2_{\boxtimes}(\Omega\times I; \mathcal{D}) 
    \rightarrow 
    L^2_{\boxtimes}(\Omega\times I; \mathcal{D}),
    \\
    &\underline{\boldsymbol z}
    \mapsto 
    [U^{\underline{\boldsymbol\alpha}}\underline{\boldsymbol z}] : (\omega,x) \mapsto
    \Big(\int_I w(x,y)K\left(\alpha^y_{t}(\omega), X_{t-}^{\underline{\boldsymbol \alpha},\underline{\boldsymbol z}, y}(\omega)\right)\lambda(dy)\Big)_{t\in[0,T]}
\end{aligned}
\end{equation}
where $\underline{\boldsymbol X}^{\underline{\boldsymbol\alpha}, \underline{\boldsymbol z}}\in \mathcal{L}_E$ is the solution to \eqref{eq:linear_dynamics_intro} characterized in Proposition~\ref{prop:X}.

\begin{lemma}
\label{prop:fp-U}
Let Condition~\ref{eq:cond_on_q} hold. For each $\underline{\alpha} \in \underline{\mathbb{A}}$ the mapping $U^{\underline{\boldsymbol\alpha}}$ has a unique fixed point in $L^2_\boxtimes(\Omega\times I; \mathcal{D})$.
\end{lemma}
Denoting the fixed point by $\underline{\boldsymbol Z}^{\underline{\boldsymbol\alpha}}$, the next lemma uses the Exact Law of Large Numbers \cite{sun2006exact} to guarantee that $\underline{\boldsymbol Z}^{\underline{\boldsymbol\alpha}}$ is $\mathcal{C}$-valued $\mathbb{P}\boxtimes\lambda$-a.s. for each $\underline{\boldsymbol\alpha}\in \underline{\mathbb{A}}$.

\begin{lemma}
\label{lemma:Uz-is-cont}
Let Condition \ref{eq:cond_on_q} and \ref{cond:K_lip} hold, and let $\underline{\alpha}\in \underline{\mathbb{A}}$. 
Then $[U^{\underline{\boldsymbol\alpha}}\underline{\boldsymbol z}] \in L^2_\boxtimes(\Omega\times I; \mathcal{C})$ for each $\underline{\boldsymbol z}\in L^2_\boxtimes(\Omega\times I; \mathcal{D})$.
\end{lemma}
This proves the part (i) of the theorem. Part (ii) and (iii) can be shown along the same lines of proof as is used in \cite[Thm. 2]{aurell2021stochastic}

\subsubsection{Proof of Lemma~\ref{prop:fp-U}}

Let $\underline{\boldsymbol z},\underline{\boldsymbol{\widetilde z}} \in L^2_\boxtimes(\Omega\times I; \mathcal{D})$. We have for $(\omega,x)\in\Omega\times I$:
\begin{equation}
    \label{eq:ineq-for-uz-1}
    \|U^{\underline{\boldsymbol\alpha}}\underline{\boldsymbol z}(\omega,x) - U^{\underline{\boldsymbol\alpha}}\underline{\boldsymbol{\widetilde z}}(\omega,x)\|_T^2
    \leq
     C\int_I \|\boldsymbol X^{\underline{\boldsymbol\alpha},\underline{\boldsymbol z},y}(\omega)
     -
     \boldsymbol X^{\underline{\boldsymbol\alpha},\underline{\boldsymbol{\widetilde{z}}},y}(\omega)\|_T^2\lambda(dy).
\end{equation}
Following similar lines of proof as in Proposition~\ref{prop:X}, get the initial $\mathbb{P}\boxtimes\lambda$-a.e. estimates
\begin{equation}
\label{eq:diff-in-state}
\begin{aligned}
    &
    |X^{\underline{\boldsymbol\alpha},\underline{\boldsymbol z},y}_t(\omega)-X^{\underline{\boldsymbol\alpha},\underline{\boldsymbol {\widetilde z}},y}_t(\omega)|^2
    \leq
    C\Bigg(
    \int_0^t\sum_{k=-n+1}^{n-1}|k||\kappa^y_s(X_{s-}^{\underline{\boldsymbol \alpha},\underline{\boldsymbol z}, y}(\omega),k, \alpha^y_s(\omega), z^y_{s-}(\omega)) 
    \\
    &- \kappa^y_s(X_{s-}^{\underline{\boldsymbol \alpha},\underline{\boldsymbol{\widetilde z}}, y}(\omega),k, \alpha^y_s(\omega), \widetilde z^y_{s-}(\omega))|ds\Bigg)^2
    + 
    C|M^{\underline{\boldsymbol\alpha},\underline{\boldsymbol z},y}_t(\omega) - M^{\underline{\boldsymbol\alpha},\underline{\boldsymbol{\widetilde{z}}},y}_t(\omega)|^2
    \end{aligned}
\end{equation}
and
\begin{align*}
    &\langle    \boldsymbol M^{\underline{\boldsymbol\alpha},\underline{\boldsymbol z},x}(\omega) - \boldsymbol M^{\underline{\boldsymbol \alpha},\underline{\boldsymbol{\widetilde z}},x}(\omega) \rangle_t
    \\
    &
    \leq 
    \int_0^t\sum_{k=-n+1}^{n-1}|k|^2|\kappa^y_s(X_{s-}^{\underline{\boldsymbol \alpha}(\omega),\underline{\boldsymbol z}, y},k, \alpha^y_s(\omega), z^y_{s-}(\omega)) - \kappa^y_s(X_{s-}^{\underline{\boldsymbol \alpha},\underline{\boldsymbol{\widetilde z}}, y}(\omega),k, \alpha^y_s(\omega), \widetilde z^y_{s-}(\omega))|ds.
\end{align*}
In view of Condition~\ref{eq:cond_on_q}
\begin{equation}
\label{eq:diff-in-quad-var}
\begin{aligned}
     &|X^{\underline{\boldsymbol\alpha},\underline{\boldsymbol z},x}_t(\omega)-X^{\underline{\boldsymbol\alpha},\underline{\boldsymbol {\widetilde z}},x}_t(\omega)|^2
    \\
    &\leq C\int_0^t\left(
    \|\boldsymbol X^{\underline{\boldsymbol\alpha},\underline{\boldsymbol z}, x}(\omega) - \boldsymbol X^{\underline{\boldsymbol\alpha},\underline{\boldsymbol{\widetilde z}}, x}(\omega)\|_s^2 + |z^x_{s-}(\omega) - \widetilde z^x_{s-}(\omega)|^2
    \right)ds,\quad \mathbb{P}\boxtimes\lambda\text{-a.s.}
    \end{aligned}
\end{equation}
Taking expectation in \eqref{eq:ineq-for-uz-1} we get (using the Fubini property) that
\begin{equation*}
    \mathbb{E}^\boxtimes\left[\|U^{\underline{\boldsymbol\alpha}}\underline{\boldsymbol z} - U^{\underline{\boldsymbol\alpha}}\underline{\boldsymbol{\widetilde z}}\|_T^2\right] \leq C\mathbb{E}^\boxtimes\left[\|\underline{\boldsymbol X}^{\underline{\boldsymbol\alpha},\underline{\boldsymbol z}} - \underline{\boldsymbol X}^{\underline{\boldsymbol \alpha}, \underline{\boldsymbol{\widetilde z}}}\|_T^2\right].
\end{equation*}
Then, by \eqref{eq:diff-in-state}, Doob's inequality, and \eqref{eq:diff-in-quad-var} we get
\begin{equation*}
    \mathbb{E}^\boxtimes\left[\|\underline{\boldsymbol X}^{\underline{\boldsymbol\alpha},\underline{\boldsymbol z}} - \underline{\boldsymbol X}^{\underline{\boldsymbol \alpha}, \underline{\boldsymbol{\widetilde z}}}\|_T^2\right]
    \leq
    C\int_0^T\mathbb{E}^\boxtimes\left[
    \|\underline{\boldsymbol X}^{\underline{\boldsymbol\alpha},\underline{\boldsymbol z}} - \underline{\boldsymbol X}^{\underline{\boldsymbol\alpha},\underline{\boldsymbol{\widetilde z}}}\|_s^2 + |\underline z_{s-} - \underline{\widetilde z}_{s-}|^2 \right]ds.
\end{equation*}
Hence, after one use of Gronwall's inequality, we have that
\begin{equation*}
        \mathbb{E}^\boxtimes\left[\|U^{\underline{\boldsymbol\alpha}}\underline{\boldsymbol z} - U^{\underline{\boldsymbol\alpha}}\underline{\boldsymbol{\widetilde z}}\|_T^2\right] 
        \leq
        C\int_0^T \mathbb{E}^\boxtimes\left[\|\underline{\boldsymbol z} - \underline{\boldsymbol{\widetilde z}}\|^2_s\right] ds
\end{equation*}
and we conclude the proof in the same way as in Proposition~\ref{prop:X}.

\subsubsection{Proof of Lemma~\ref{lemma:Uz-is-cont}}

Consider the function
$$
[0,T]\ni t \mapsto [U^{\underline{\boldsymbol\alpha}}\underline{\boldsymbol z}]_t(\omega,x) := \int_I w(x,y) K(a^y_t, X^{\underline{\boldsymbol\alpha},\underline{\boldsymbol z},y}_{t-})\lambda(dy),\quad (\omega,x)\in \Omega\times I.
$$
Recall that $\underline{ \alpha} \in \underline{\mathbb{A}}$ means that $\alpha^y_t = \underline{\alpha}(y,t,X^{\underline{\boldsymbol \alpha},\underline{\boldsymbol z}, y}_{t-})$ for $y\in I$ and $t\in [0,T]$. It follows from the Lipschitz assumption on $K$, compactness of $A$, and definition of $\mathbb{A}$ that
\begin{align*}
    &
    |K(\alpha^y_t, X^{\underline{\boldsymbol \alpha},\underline{\boldsymbol z}, y}_{t-}) - 
    K(\alpha^y_s, X^{\underline{\boldsymbol \alpha},\underline{\boldsymbol z}, y}_{s-})| 
    \\
    & \leq 
    C\Big(
    |\underline{\alpha}(y,t,X^{\underline{\boldsymbol \alpha},\underline{\boldsymbol z}, y}_{s-}) - 
    \underline{\alpha}(y,s,X^{\underline{\boldsymbol \alpha},\underline{\boldsymbol z}, y}_{s-})| 
    + 
    |X^{\underline{\boldsymbol \alpha},\boldsymbol z, y}_{t-} - X^{\underline{\boldsymbol \alpha},\boldsymbol z, y}_{s-}|
    \Big),\quad s,t\in[0,T].
\end{align*}
Let $t_j\in[0,T]$, $j\in \mathbb{N}$, be a sequence converging to $t^*\in [0,T]$. Without loss of generality, assume that the sequence is non-decreasing. Recall that $q_{\max}$ denotes the uniform upper bound for the intensity rates, see Condition~\ref{eq:cond_on_q}.\ref{eq:cond_on_q-bdd}. For $(\omega,x)\in (\Omega\times I)$,
\begin{align*}
    &
    \left|\int_I w(x,y)\left( X^{\underline{\boldsymbol\alpha},\underline{\boldsymbol z}, y}_{t_j-}(\omega) - X^{\underline{\boldsymbol\alpha},\underline{\boldsymbol z}, y}_{t^*-}(\omega)\right)\lambda(dy)\right|
    \\
    &\leq C\int_I \sum_{k=-n+1}^{n-1} \left|\int_{\mathbb{R}\times [t_j,t^*)} 1_{[0, \kappa^y_s(X^{\underline{\boldsymbol\alpha},\underline{\boldsymbol z}, y}_{s-}(\omega), k, \alpha^y_s(\omega), z^y_{s-}(\omega))]}(u)N^y_k(\omega,du\otimes ds) \right|\lambda(dy)
    \\
    &\leq
    C\int_I \widetilde{N}^y_j(\omega) \lambda(dy),
\end{align*}
where $\widetilde{N}^y_j := \sum_{k=-n+1}^{n-1} N^y_k([0,q_{\max}]\times [t_j,t^*))$ is for each $y\in I$ a Poisson-distributed random variable with intensity $(2n-1)q_{\max}|t^* - t_j|$, since the summands $N^y_k([0,q_{\max}]\times [t_j,t^*))$, $k=-n+1,\dots, n-1$ are independent Poisson-distributed with intensity $q_{\max}|t^* - t_j|$. Moreover, $(\widetilde N^y_j)_{y\in I}$ are e.p.i., so by the Exact Law of Large Numbers
\begin{align*}
    \int_I \widetilde{N}^y_j(\omega) \lambda(dy)
    =
    \int_I \mathbb{E}\left[ \widetilde{N}^y_j \right]\lambda(dy)
    = (2n-1)q_{\max}|t^*-t_j|,\quad \mathbb{P}\text{-a.e. } \omega\in \Omega.
\end{align*}
Hence, by the boundedness of $A$, it holds for all $x\in I$ and $\mathbb{P}$-a.e. $\omega\in \Omega$ that 
\begin{align*}
    & \lim_{j\rightarrow\infty}\left|[U^{\underline{\boldsymbol\alpha}}\underline{\boldsymbol z}]_{t_j}(\omega,x) - [U^{\underline{\boldsymbol\alpha}}\underline{\boldsymbol z}]_{t^*}(\omega,x)\right|
    \\
    &\leq 
    \lim_{j\rightarrow\infty}C\int_I\left(|\underline{\alpha}(y,t_j, X_{t_j-}^{\underline{\boldsymbol\alpha},\underline{\boldsymbol z}},y) - \underline{\alpha}(y, t^*, X_{t_j-}^{\underline{\boldsymbol\alpha}, \underline{\boldsymbol z}, y})| + 
    \widetilde{N}^y_j(\omega)\right)\lambda(dy)
    \rightarrow 0,
\end{align*}
and in particular $[U^{\underline{\boldsymbol\alpha}}\underline{\boldsymbol z}] (\omega,x) \in \mathcal{C}$ for $\mathbb{P}\boxtimes\lambda$-a.e. $(\omega,x)\in \Omega\times I$. Measurability and square-integrability of $U^{\underline{\boldsymbol\alpha}}\underline{\boldsymbol z}$ follows by the same lines of proof as Lemma 1 in \cite{aurell2021stochastic}.

\bibliographystyle{siam}

\section*{Additional Tables}

\begin{table}[H]
\small
\centering
\begin{tabular}{c|ccccccc}
\textbf{Parameters}    & \multicolumn{1}{c}{$T$} & \multicolumn{1}{c}{$\lambda^{\mathsf S}$} & \multicolumn{1}{c}{$\lambda^{\mathsf I}$}& \multicolumn{1}{c}{$\lambda^{\mathsf R}$} &  \multicolumn{1}{c}{$c_\lambda$} & \multicolumn{1}{c}{$c_I$} & \multicolumn{1}{c}{$c_D$}\\[1mm] 
\hline 
&&&&&&&\\[-1mm]
\textbf{Policy 1}   &  200  & [1.0, 1.0, 1.0, 1.0]   & [1.0, 1.0, 1.0, 1.0] & 1.0 & 10  & 1 & 1\\
&&&&&&&\\[-1mm]
\textbf{Policy 2}   &  200  & [1.0, 1.0, 1.0, 1.0]   & [0.5, 0.5, 0.5, 0.5] & 1.0 & 10  & 1 & 1\\
&&&&&&&\\[-1mm]
\textbf{Policy 3}   &  200  & [0.5, 1.0, 1.0, 0.5]   & [0.5, 0.5, 0.5, 0.5] & 1.0 & 10  & 1 & 1\\
&&&&&&&\\[-1mm]
\textbf{Policy 4}   &  200  & [0.5, 0.5, 0.5, 0.5]   & [0.5, 0.5, 0.5, 0.5] & 1.0 & 10  & 1 & 1\\
&&&&&&&\\[-1mm]
\end{tabular}
\caption{\small Parameters used in the experiment with age-groups specific lockdowns: $\lambda^{\mathsf S}$ and $\lambda^{\mathsf I}$ vary between the age groups. The lockdown is imposed by decreasing the $\lambda^{\mathsf S}$ and $\lambda^{\mathsf I}$ values of the corresponding blocks.}
\end{table}

\begin{table}[H]
\small
\centering
\begin{tabular}{c|cccccccc}
\textbf{Parameters}    & \multicolumn{1}{c}{$T$} & \multicolumn{1}{c}{$\gamma$} &\multicolumn{1}{c}{$\lambda^{\mathsf S}$} & \multicolumn{1}{c}{$\lambda^{\mathsf I}$}& \multicolumn{1}{c}{$\lambda^{\mathsf R}$} &  \multicolumn{1}{c}{$c_\lambda$} & \multicolumn{1}{c}{$c_I$} & \multicolumn{1}{c}{$c_D$}\\[1mm] 
\hline 
&&&&&&&&\\[-1mm]
\textbf{All policies}   & 40   & 0.1   & 1.0 & 0.9 & 1.0 & 10 & 1 & 1
\end{tabular}
\caption{\small Parameters used in the experiment with different cities}
\end{table}

\begin{table}[H]
\small
\centering
\begin{tabular}{c|ccccccccccccc}
\textbf{Param.}    & \multicolumn{1}{c}{$T$} & \multicolumn{1}{c}{$\beta$} & \multicolumn{1}{c}{$\gamma$}& \multicolumn{1}{c}{$\epsilon$} &  \multicolumn{1}{c}{$\rho$}& \multicolumn{1}{c}{$c_\lambda$} & \multicolumn{1}{c}{$c_I$} & \multicolumn{1}{c}{$c_D$} & \multicolumn{1}{c}{$g$} & \multicolumn{1}{c}{$\lambda^{\mathsf S}$} & \multicolumn{1}{c}{$\lambda^{\mathsf E}$} & \multicolumn{1}{c}{$\lambda^{\mathsf I}$} & \multicolumn{1}{c}{$\lambda^{\mathsf R}$} \\[1mm] 
\hline 
&&&&&&&&&\\[-1mm]
\textbf{Values}   & 40   & 0.2   & 0.1 & 0.2 & 0.95 & 10 & 1 & 1  & -0.2 &  1.0 &  1.0 &  0.9  &  1.0
\end{tabular}
\caption{\small Parameters in the SEIRD model experiments with power law graphon.}
\end{table}

\end{document}